\documentclass{article}

\makeatletter

\@addtoreset{equation}{section}
\makeatother

\usepackage{amsmath}
\usepackage{amssymb}
\usepackage{amsthm}
\usepackage{graphics}
\usepackage{epsfig}
\usepackage{fullpage}
\usepackage{authblk}

\numberwithin{equation}{section}

\usepackage[dvipdfmx, colorlinks=true, allcolors=blue]{hyperref}

\def\a{{\alpha}}
\def\b{{\beta}}
\def\e{{\varepsilon}}
\def\k{{\kappa}}

\def\pa{{\partial}}

\def\r{{\mathbb{R}}}
\def\z{{\mathbb{Z}}}

\newtheorem{theorem}{Theorem}

\newtheorem{lemma}{Lemma}

\newtheorem{remark}{Remark}

\begin{document}

\title{Rigorous Construction of Stop-and-Go Waves in the Optimal Velocity Model via a Difference-Differential Equation}

\author[1]{Kota Ikeda\thanks{Corresponding author. Email: ikeda@meiji.ac.jp}\thanks{These authors contributed equally to this work.}}
\author[2]{Tomoyuki Miyaji\thanks{Email: miyaji.tomoyuki.4m@kyoto-u.ac.jp}\thanks{These authors contributed equally to this work.}}

\affil[1]{School of Interdisciplinary Mathematical Sciences, Meiji University, Nakano, Tokyo, Japan}
\affil[2]{Department of Mathematics, Kyoto University, Kyoto, Kyoto, Japan}

\maketitle

\begin{abstract}
Nonlinear wave phenomena such as stop-and-go traffic patterns are widely observed in vehicular flow but remain challenging to describe within a rigorous mathematical framework. Motivated by this, we investigate nonlinear wave structures in the optimal velocity (OV) model, which is a fundamental microscopic traffic flow model describing the car-following dynamics on a circuit. Using a traveling-wave formulation for vehicle headways, we reduce the original ordinary differential system to a difference-differential equation. We focus on steep OV functions approaching a step function, which generate sharp transition layers in the headway profile. In the singular limit, we explicitly construct heteroclinic transition layer solutions connecting two uniform traffic states. Motivated by related solvable queueing models in the literature, we rigorously prove the existence of heteroclinic traveling waves for sufficiently steep OV functions. We further establish the existence of homoclinic solutions arising from the interaction of increasing and decreasing transition layers, and derive a necessary condition for the amplitude parameter for their existence. To construct periodic stop-and-go waves on a circuit, we impose a global constraint reflecting the conservation of the total road length. Within this constrained framework, we prove the existence of large-period periodic solutions comprising alternating transition layers and quasi-uniform states. Beyond local bifurcation analysis, these results establish a rigorous foundation for nonlinear congestion waves. Furthermore, they contribute to the validation of car-following models and the design of control strategies to mitigate congestion.

\vspace{1em}
\end{abstract}


\section{Introduction}\label{sec_Introduction}

The rapid proliferation of vehicles on road networks, 
including highways and urban streets, has raised significant economic, social, 
and environmental concerns---specifically regarding energy consumption, 
traffic safety, and pollution \cite{bellomo2002mathematical}.
Consequently, traffic flow modeling has become a highly active research field, 
attracting interest from both applied mathematicians and engineers. 
Understanding and predicting traffic dynamics are crucial for rational planning 
and effective management of vehicular flow, particularly urban congestion.
Transportation efficiency continues to pose a pressing global challenge.
Despite significant advancements, the inherent complexity of traffic systems 
renders developing a comprehensive mathematical framework that accurately 
captures the dynamic and nonlinear nature of vehicular movement difficult. 
This complexity necessitates a multidisciplinary approach that integrates 
theoretical modeling with experimental data, highlighting the need 
for further theoretical advancements to enhance predictive capabilities 
and optimize traffic management strategies \cite{piccoli2009vehicular}.

Mathematical modeling plays a crucial role in understanding and optimizing traffic 
flow by describing vehicle dynamics through various approaches 
\cite{bellomo2002mathematical}. 
A key concept is the fundamental diagram, i.e., the empirical relationship linking 
traffic speed (or flow) to vehicle density. 
Here, the density is closely related to the headway, 
which denotes the distance (or time gap) between successive vehicles. 
Traffic flow models are traditionally categorized into microscopic and macroscopic 
models \cite{konishi2000decentralized}. 
Microscopic models such as car-following models describe individual vehicle behavior,
whereas macroscopic models, inspired by fluid dynamics, 
treat traffic as a continuous medium \cite{van2015genealogy}. 
As a fundamental microscopic car-following framework, the optimal velocity (OV) 
model proposed by \cite{bando1994structure} has been widely applied. 
In this model, each driver adjusts their acceleration toward 
a desired velocity determined by the headway to the leading vehicle. 
Specifically, the model is given by
\begin{equation}\label{eq_OV}
  \ddot x_n = a (V (\Delta x_n) - \dot x_n), \quad n = 1, \ldots, N
\end{equation}
in a circuit of length $L$, where $x_n = x_n (\tau)$ is the position of vehicle $n$ 
at time $\tau$, and $a$ is a positive constant related to driver sensitivity. 
We denote the distance between vehicle $n$ and its preceding vehicle $n+1$ 
by $\Delta x_n = x_{n+1} - x_n$ for $n = 1, \ldots, N$, where $x_{N+1} \equiv x_1 + L$.
The dot $\dot {}$ denotes the derivative with respect to $\tau$. 
The OV function $V$ explicitly links spacing, speed, and flow, 
naturally connecting the model to the fundamental empirical diagram 
and enabling the interpretation of capacity drops and metastable states 
\cite{bando1995phenomenological}. 
It is typically given by
\begin{equation}\label{eq_OVfunction_tanh}
  V (x) = \dfrac{V_0}{2} (\tanh (\beta (x - l)) + M)
\end{equation}
for $\beta, V_0, l > 0$ and $M \in \r$ as observed in \cite{bando1995phenomenological}. 
Formulated as a system of ordinary differential equations, this model successfully 
reproduces the spontaneous emergence of stop-and-go waves, 
playing a key role in advancing microscopic traffic modeling 
\cite{kikuchi2000asymmetric}. 
Extensive analysis of its uniform equilibrium reveals that homogeneous flow becomes 
unstable beyond a critical sensitivity threshold, resulting in oscillatory behavior 
\cite{bando1994structure, bando1995dynamical, helbing1998generalized}.

The model has been extended in multiple directions, including the Full Velocity 
Difference model incorporating relative velocity terms \cite{jiang2001full}, 
multi-anticipative, and cooperative driving models \cite{treiber2003memory}, 
and variants that account for driver heterogeneity, reaction time delays, 
or automated-vehicle control. 
These developments reflect a broader evolution from simple dynamic systems 
toward realistic representations of human and automated driving behavior, 
exemplified by the Intelligent Driver Model,
whose congested states and fundamental characteristics have been 
investigated using large-scale detector data obtained from German freeways 
\cite{treiber2000congested, treiber2000microscopic}. 
In addition to continuous models, cellular automata approaches offer an alternative 
microscopic paradigm with well-documented properties \cite{maerivoet2005cellular}. 
Comprehensive reviews have summarized the mathematical characteristics of the OV model 
itself \cite{sugiyama2023dynamics} as well as broader perspectives on car-following 
and traffic modeling \cite{brackstone1999car, maerivoet2005cellular, zhang2024car}. 
Thus, despite extensive analytical efforts, (\ref{eq_OV}) and its descendants 
remain foundational tools for interpreting traffic dynamics and informing 
modeling practices in contemporary transportation research.
Conversely, although analytical approaches have been considered,
they predominantly focus only on identifying the bifurcation points that cause
Hopf bifurcation through the stability analysis of uniform flows; 
the other parts are analyzed using numerical simulations \cite{gasser2004bifurcation}.

To analytically capture the characteristic solutions representing traffic 
congestion---specifically, the stop-and-go waves---we focus on the fact 
that fully developed congested states exhibit distinct periodic fluctuations 
in vehicle headway. 
Motivated by this observation, we formulate a governing equation for such 
oscillatory behavior. 
Consequently, we are interested in a solution $x_n (t)$ in (\ref{eq_OV}), 
for which $\Delta x_n$ is a periodic function with period $\omega$,
and is represented by $\Delta x_n (\tau) = u (c \tau + n-1)$ 
for a function $u = u (t)$ with $t = c \tau + n-1$ and $c = N / \omega$.
Since $\Delta x_n (\tau) = \Delta x_{n} (\tau + \omega)$, 
$u$ must be a periodic function with period $N$ because
$u (c \tau + n - 1) = u (c (\tau + \omega) + n - 1) = u (c \tau + n - 1 + N)$.
Since $\Delta x_{n+1} (\tau) = u (c \tau + n)$, we have 
\[
c^2 u'' (c \tau + n-1) = a (V (u (c \tau + n)) - V (u (c \tau + n-1))) 
- a c u' (c \tau + n-1),
\]
where the prime $'$ denotes a derivative with respect to its argument.
Then, $u$ is governed by 
\begin{equation}\label{eq_u}
 c^2 u'' (t) + a c u' (t) = a (V (u (t + 1)) - V (u (t))).
\end{equation}
Assuming (\ref{eq_OVfunction_tanh}), we carried out numerical simulations 
for (\ref{eq_u}).
The numerical results indicate the existence of a pair of 
$c > 0$ and a periodic solution $u$ with period $N$ (see Fig.~\ref{fig_periodic}).
Throughout this study, the pair $(u, c)$ is referred to as the solution to (\ref{eq_u}).
When $N$ increases, the state of $u$ is divided into four parts.
More precisely, there exist two intervals $I_i, I_d \subset [0, N]$ 
independent of $N$ such that $u (t) |_{I_i}$ and $u (t) |_{I_d}$ monotonically 
increase and decrease and are close to the transition layers connecting 
some constants, denoted by $u_1$ and $u_2$.
In the other regions, $u$ is close to either $u_1$ or $u_2$.
These numerical observations confirm that this specific profile of $u$ effectively 
represents the traffic congestion in the original OV model (\ref{eq_OV}).
Therefore, successfully capturing this solution $(u, c)$ to (\ref{eq_u}) enables us 
to analytically construct the target characteristic solutions.


\begin{figure}[h]
 \begin{center}
  \begin{tabular}{c}
   \includegraphics[width=7cm]{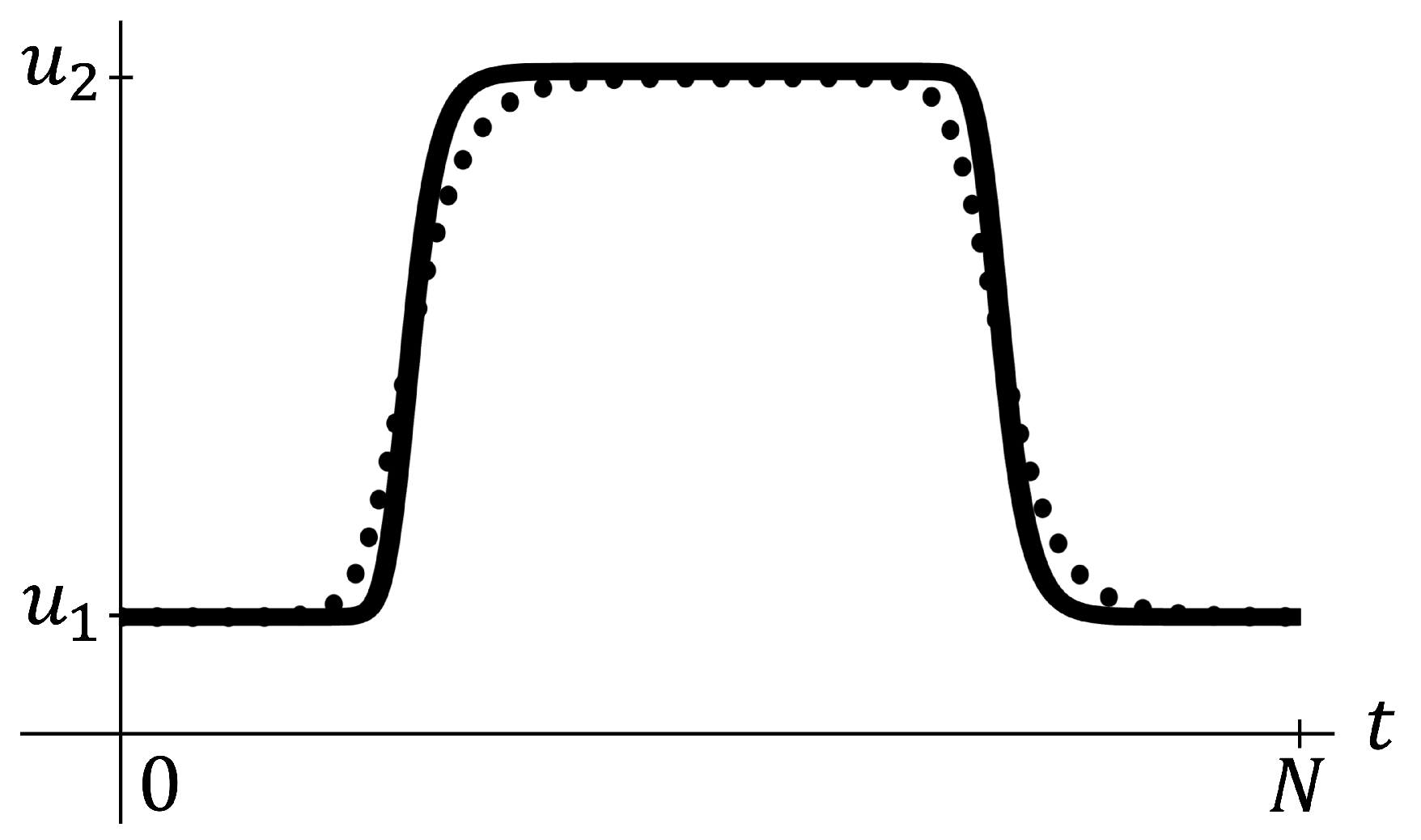}
  \end{tabular}
  \caption{
  \label{fig_periodic}
  Periodic solutions of (\ref{eq_u}) on the interval $[0, N]$. 
  The dotted and solid lines represent numerical results for $N = 20$ 
  and $N = 40$, respectively. 
  The OV function is given by (\ref{eq_OVfunction_tanh}) with parameters 
  $V_0 = 0.0336$, $\beta = 2 / 0.0223$, $l = 0.025$, $M = 0.913$, and $a = 1.6$.
  }
 \end{center}
\end{figure}


Consider the existence of solutions to (\ref{eq_u})
associated with $u (t) |_{I_i}$ and $u (t) |_{I_d}$.
Throughout this study, we make the following assumptions
on the OV function $V \in C^2 (\r)$:
\begin{itemize}
 \item[(A$_1$)]
	      $V$ depends on a parameter $\b > 0$ and is close to $V_0 H (x - l) + V_1$
	      for $l, V_0 > 0$ and $V_1 \in \r$ 
	      where $H = H (x)$ is the unit step function defined by 
	      $H (x) = 0$ if $x < 0$ and $H (x) = 1$ if $x \ge 0$.
	      More precisely, there are $\a, C_0 > 0$ independent of $\b$ such that 
	      \[
	      \begin{aligned}
	       & |V (x) - (V_0 H (x - l) + V_1)| \le C_0 e^{- \a \b |x - l|}, \\
	       & |V' (x)| \le C_0 \b e^{- \a \b |x - l|},
	       \quad |V'' (x)| \le C_0 \b^2 e^{- \a \b |x - l|}	       
	      \end{aligned}
	      \]
	      in $x \in \r$.
\end{itemize}
The OV function given by (\ref{eq_OVfunction_tanh})
satisfies all the aforementioned assumptions by setting $V_1 = V_0 (M - 1) / 2$.
As shown in (A$_1$), $V$ is assumed to converge to a step function
as the limit for $\b \to \infty$.
Based on this assumption, we expect that $u (t) |_{I_i}$ and $u (t) |_{I_d}$ 
can be approximated using solutions with transition layers that satisfy the following profile equation:
\begin{equation}\label{eq_u_infty}
 c^2 u'' (t) + a c u' (t) = a V_0 (H (u (t + 1) - l) - H (u (t) - l)),
  \quad t \in \r.
\end{equation}
It is easy to prove that if $0 < \eta < V_0 / a$, (\ref{eq_u_infty}) has 
a unique solution $c = c_0 > 0$ and $u = u_i \in C^1 (\r)$ that satisfies 
\begin{equation}\label{eq_BC}
 u (-\infty) < l,
  \quad u (0) = l, 
  \quad u (\infty) = l + \eta,
\end{equation}
where $u (\pm \infty) \equiv \lim_{t \to \pm \infty} u (t)$.
Here we note that $u_2 = l + \eta$.
Because of translation invariance, we assume the second condition in (\ref{eq_BC})
without loss of generality.
More precisely, $u_i = u_i (t)$ and $c_0$ can be explicitly expressed as
\begin{equation}\label{eq2_ui_explicit}
 u_i (t) = 
  \left\{
   \begin{aligned}
    & l + \eta + \dfrac{V_0}{a} \log ( 1 - \dfrac{a}{V_0} \eta ), 
    & & t < -1, \\
    & l + \eta - \left( \dfrac{V_0}{a} (1 - e^{- \k t}) 
    - V_0 \dfrac{t}{c_0} + \eta e^{- \k t} \right), 
    & & -1 \le t < 0, \\
    & l + \eta - \eta e^{- \k t}, 
    & & t \ge 0,
   \end{aligned}
  \right.
\end{equation}
and
\begin{equation}\label{eq2_cs}
 c_0 = - \dfrac{a}{\log ( 1 - \dfrac{a}{V_0} \eta )},
\end{equation}
where $\k \equiv a / c_0$.
Similarly, (\ref{eq_u_infty}) has a unique solution 
$c = c_0 > 0$ and $u = u_d \in C^1 (\r)$ that satisfies
\begin{equation}\label{eq_BC_2}
 u (-\infty) > l,
  \quad u (0) = l, 
  \quad u (\infty) = l - \eta.
\end{equation}
We can easily verify that  $2 l - u_i (t)$ satisfies
the solution to (\ref{eq_u_infty}) with (\ref{eq_BC_2}),
that is, $u_d (t) = 2 l - u_i (t)$ because 
\begin{equation}\label{eq_H_symmetry}
 \begin{aligned}
  H (u_d (t + 1) - l) - H (u_d (t) - l)
  & = H (l - u_i (t + 1)) - H (l - u_i (t)) \\
  & = - (H (u_i (t + 1) - l) - H (u_i (t) - l)).
 \end{aligned} 
\end{equation}
To emphasize the $\eta$-dependencies of $c_0$, $u_i$, and $u_d$,
we may denote $c_0 (\eta)$, $u_i (t; \eta)$, and $u_d (t; \eta)$.
This study aims to prove the existence of a solution $(u, c)$ in (\ref{eq_u})
approximated by $(u_i, c_0)$ and $(u_d, c_0)$.

We first state the existence of a solution to (\ref{eq_u}) in $\r$
that satisfies (\ref{eq_BC}) or (\ref{eq_BC_2}).
We call the corresponding solution the {\it heteroclinic solution}.
All constants are stationary solutions of (\ref{eq_u}).


\begin{theorem}\label{thm_1}
 Let $0 < \eta < V_0 / a$ be arbitrary.
 If $\b$ is sufficiently large, then there exist $C > 0$ independent of $\b$
 and a solution $(u, c)$ of (\ref{eq_u}) and (\ref{eq_BC}) (resp. (\ref{eq_BC_2}))
 such that $|c - c_0| \le C \b^{-1/2}$ and 
 $|u (t) - u_i (t)| \le C \b^{-1/2}$ (resp. $|u (t) - u_d (t)| \le C \b^{-1/2}$)
 in $t \in \r$.
\end{theorem}


Next, we state that there exists a {\it homoclinic solution} $u$ to (\ref{eq_u}), 
which is defined on $\r$ and satisfies $u (-\infty) = u (\infty)$.
We now provide the necessary condition for $\eta$
in the existence of a homoclinic solution $u$ that satisfies
\begin{equation}\label{eq_formalsol}
 u (t) \sim 
  \left\{
   \begin{aligned}  
    & u_i (t),
    & & t \sim 0, \\
    & u_d (t - t_*),
    & & t \sim t_*
   \end{aligned}
  \right.
\end{equation}
or
\begin{equation}\label{eq_formalsol2}
 u (t) \sim 
  \left\{
   \begin{aligned}  
    & u_d (t),
    & & t \sim 0, \\
    & u_i (t - t_*),
    & & t \sim t_*
   \end{aligned}
  \right.
\end{equation}
for a given $t_* > 0$.
Since $u_i (-\infty)$ must be consistent with $u_d (\infty)$, we obtain
\[
l + \eta + \dfrac{V_0}{a} \log ( 1 - \dfrac{a}{V_0} \eta ) = l - \eta,
\]
from which $\eta$ must satisfy 
\begin{equation}\label{eq_necessry_cond}
 2 \dfrac{a}{V_0} \eta = - \log ( 1 - \dfrac{a}{V_0} \eta ).
\end{equation}
Through elementary calculations, we can easily show that 
$2 X = - \log (1 - X)$ has a unique positive solution $X \approx 0.797$.
Then, $\eta_* \equiv V_0 X / a$ is a solution to (\ref{eq_necessry_cond}).
This condition also leads to $u_i (\infty) = u_d (-\infty)$
because of $u_d (t) = 2 l - u_i (t)$.
Hence, $\eta$ should be close to $\eta_*$ if there exists 
a homoclinic solution of (\ref{eq_u}) that satisfies either 
(\ref{eq_formalsol}) or (\ref{eq_formalsol2}).
Hereafter, we denote $c_* = c_0 (\eta_*)$.


\begin{theorem}\label{thm_4}
 For any $\e > 0$, let $R > 0$ be sufficiently large. 
 For such $\e$ and $R$, let $t_* > 0$ be a sufficiently large 
 constant independent of $\b$. 
 If $\b$ is sufficiently large, then (\ref{eq_u}) has a homoclinic solution 
 $(u, c)$ for some $\eta > 0$ that satisfies $u(0) = u(t_*) = l$. 
 Moreover, there exists $C > 0$ independent 
 of $\b, \e, R$, and $t_*$ such that $|c - c_*| \le C \e$, $|\eta - \eta_*| \le C \e$, 
 $|u (t) - u_i (t; \eta_*)| \le C \e$ 
 (resp. $|u (t) - u_d (t; \eta_*)| \le C \e$) in $t \le t_* - 2 R$,
 and $|u (t) - u_d (t - t_*; \eta)| \le C \e$ 
 (resp. $|u (t) - u_i (t - t_*; \eta)| \le C \e$) in $t_* - 2 R < t$.
\end{theorem}


Finally, we state the existence of a periodic solution to (\ref{eq_u}) 
and (\ref{eq_formalsol}) for period $N$.
However, in contrast to the heteroclinic and homoclinic solution cases, 
constructing periodic solutions to (\ref{eq_u}) using perturbation methods is difficult.
Therefore, we use the intrinsic properties of the OV model (\ref{eq_OV}) 
and introduce a modified mathematical model derived from (\ref{eq_u}).

The periodic solution $u$ in (\ref{eq_u}) must satisfy a certain constraint.
Assuming that all vehicles move in the positive direction of the circuit, 
the sum of their headways $\Delta x_n (\tau)$ for $n = 1, \ldots, N$ 
must be equal to the total length of the circuit.
Then, we have $x_{N+1} (\tau) - x_1 (\tau) = L$ for any $\tau$.
Since $\Delta x_n (\tau) = u (c \tau + n - 1)$, we set $c \tau \to t$ and see
\begin{equation}\label{eq_discrete_constraint}
  \sum_{n=1}^N u (t + n - 1) = L
\end{equation}
for any $t$.
Similar to the method followed in \cite{okamoto2024nonlinear}, 
we can show that (\ref{eq_discrete_constraint}) leads to
\begin{equation}\label{eq_continuous_constraint}
  \dfrac{1}{N} \int_0^N u (t) dt = \dfrac{L}{N}.
\end{equation}
Let $M$ be an arbitrary positive integer.
For $\Delta t = 1 / M$, we set $t \to t + k \Delta t$ for $k = 0, \ldots, M - 1$ 
in (\ref{eq_discrete_constraint}).
The sum of the resulting equation with respect to $k$ derives
\[
\sum_{l = 0}^{M N - 1} u (t + l \Delta t) \Delta t
= \sum_{k = 0}^{M-1} \sum_{n=1}^N u (t + n - 1 + k \Delta t) \Delta t = L.
\]
Taking the limit $M \to \infty$, we obtain (\ref{eq_continuous_constraint}).
From (\ref{eq_formalsol}), the left-hand side of (\ref{eq_continuous_constraint}) 
can be approximated by
\begin{equation}\label{eq_t0u0}
  \dfrac{1}{N} \int_0^N u (t) dt \approx
  (l + \eta) \dfrac{t_*}{N} + (l - \eta) \left( 1 - \dfrac{t_*}{N} \right).
\end{equation}

In the original physical setting of (\ref{eq_OV}), $L$ is a predetermined circuit 
length, which would naturally determine the parameter $t_*$ through 
the approximation (\ref{eq_t0u0}).
However, to construct the periodic solution to (\ref{eq_u}), we reverse this dependency.
We treat $t_*$ as an independent parameter of the solution and impose 
the strict equality
\begin{equation}\label{eq_const}
  \dfrac{1}{N} \int_0^N u (t) dt 
  = (l + \eta) \dfrac{t_*}{N} + (l - \eta) \left( 1 - \dfrac{t_*}{N} \right)
\end{equation}
as a primary constraint on $u$, rather than using (\ref{eq_continuous_constraint}).
Under this framework, we prove the existence of a periodic solution to (\ref{eq_u}).
Since $u$ is required to satisfy (\ref{eq_discrete_constraint}), 
the length $L$ is defined accordingly.
This ensures that (\ref{eq_continuous_constraint}) holds.
Therefore, by combining (\ref{eq_continuous_constraint}) with (\ref{eq_const}), 
the corresponding circuit length is naturally given by 
$L = (l + \eta) t_* + (l - \eta) ( N - t_* )$.


\begin{theorem}\label{thm_3}
 Assume that $t_* \neq N / 2$ satisfies 
 \begin{itemize}
  \item[(A$_2$)]
	       There exist $0 < t_1 < t_2 < 1$
	       independent of $\b$ and $N$ such that $N t_1 \le t_* \le N t_2$.
 \end{itemize}
 For any $\e > 0$, let $R > 0$ be sufficiently large. 
 If $\b$ and $N$ are sufficiently large, then (\ref{eq_u}) has a periodic solution 
 $(u, c)$ for some $\eta > 0$ satisfying $u(0) = u(t_*) = l$ and (\ref{eq_const}). 
 Moreover, there exists $C > 0$ independent 
 of $\b, \e, N$, and $R$ such that $|c - c_*| \le C \e$, $|\eta - \eta_*| \le C \e$,
 $|u (t) - u_i (t; \eta)| \le C \e$ in $- 2 R \le t \le t_* - 2 R$, and
 $|u (t) - u_d (t - t_*; \eta)| \le C \e$ in $t_* - 2 R \le t \le N - 2R$.
\end{theorem}


\begin{remark}
 In the statements of Theorems~\ref{thm_4} and \ref{thm_3}, 
 $N$ and $t_*$ can be taken independently of $\b$.
\end{remark}


We show how from the periodic solution $(u, c)$ of (\ref{eq_u}) and (\ref{eq_const})
with $c > 0$, we can  derives a solution $\{ x_n \}_{n=1}^N$ of (\ref{eq_OV}) 
using the relation $\Delta x_n (\tau) = u (c \tau + n - 1)$ and setting $x_1 (0) = x_0$.
Let $v = v (\tau; p)$ be a unique periodic solution with a period $\omega = N / c$ 
for a given continuous function $p = p (\tau)$ in 
\begin{equation}\label{eq1_genelral_vp}
 \left\{
 \begin{aligned}
  & \dot v + a v = p, \\
  & v (0; p) = \frac{1}{e^{a \omega} - 1} \int_0^\omega e^{as} p (s) ds.
 \end{aligned}
 \right.
\end{equation}
Here, we set $v (\tau) \equiv v (\tau; p)$ for $p (\tau) = a V (u (c \tau))$.
Then, $\dot x_1 (\tau)$ is given by 
the periodic solution $v (\tau)$ in (\ref{eq1_genelral_vp}).
Hence, $x_1 (\tau)$ is given by 
\[
x_1 (\tau) = x_0 + \int_0^\tau \dot x_1 (s) ds = x_0 + \int_0^\tau v (s) ds.
\]
The relationship $x_{n+1} (\tau) - x_n (\tau) = \Delta x_n (\tau) = u (c \tau + n - 1)$
enables us to reconstruct $x_{n+1}$ for $n = 1, \ldots, N-1$ using
\[
x_{n+1} (\tau) = x_1 (\tau) + \sum_{k=1}^n u (c \tau + k - 1)
= x_0 + \int_0^\tau v (s) ds + \sum_{k=1}^n u (c \tau + k - 1).
\]
This construction provides a systematic method for obtaining the full solution 
based on the reduced dynamics.

Mathematical approaches to traffic flow have been aimed at revealing
mechanisms underlying congestion formation, stability, 
and collective driving behavior.
Thus a number of analytical frameworks 
have been developed at both the macroscopic and microscopic levels.
Macroscopic traffic flow models, originating from seminal works 
\cite{lighthill1955kinematicI, lighthill1955kinematicII} 
and subsequently extended in \cite{payne1979critical}, 
rely on conservation laws to describe mass density and momentum.
They provide system-wide insights into vehicle interactions that supports 
optimization of traffic flow, congestion mitigation
and accident reductions \cite{piccoli2009vehicular}.
These models have been extensively researched, and their mathematical development 
has been comprehensively surveyed in review articles such as \cite{helbing2001traffic}. 
Despite the remarkable progress, analytical approaches for macroscopic models 
typically remain confined to neighborhoods near bifurcation points.
Rigorous mathematical results in the parameter regimes far from these 
thresholds are scarce.
In this context, a recent study analyzed the macroscopic models 
proposed by \cite{kuhne1987freeway, lee2001macroscopic}, 
successfully constructing a traveling wave solution corresponding to the 
congestion phase \cite{ikeda2025existence}, thereby providing a strict 
mathematical guarantee for the bifurcation 
structure as originally suggested by \cite{lee2004steady}. 
Analytical investigations of the microscopic traffic models yielded 
notable results for special cases.
Whitham \cite{whitham1990exact} established the exact solution 
for the Newell model \cite{newell1961nonlinear} by revealing its connection 
to the Toda lattice \cite{toda1967vibration}, a perspective further developed 
in \cite{igarashi1999toda}.
Hasebe et al. \cite{hasebe1999exact} derived another exact solution using 
elliptic functions. 
However, these successes rely heavily on the specific mathematical structure 
of the Newell model, rendering generalizing such analytical results 
to other microscopic models difficult.
Consequently, rigorous mathematical analysis remains insufficient 
for most microscopic models, including the OV model (\ref{eq_OV}).

The remainder of this paper is organized as follows.
In Section~\ref{sec_Heteroclinic}, we consider the existence of 
the heteroclinic solution of (\ref{eq_u}) under the boundary condition (\ref{eq_BC}). 
To prove Theorem~\ref{thm_1}, we set an arbitrary constant $0 < \eta < V_0 / a$ 
and require that the approximate solution $(u_i, c_0)$ is close to the true 
solution $(u, c)$. 
We also obtain a first-order approximation of the error, 
denoted by $\phi_0$ (see (\ref{eq2_phi0})). 
In fact, as seen in (\ref{eq2_definition_zeta0}), the function $\phi_0$ fails 
to be $C^1$ at $t=0$ because of the discontinuity in the step function $H$. 
To construct the true solution using an iterative method in the functional space $W$, 
we divide the interval $\r$ into two subintervals $I_1$ and $I_2$,
and solve (\ref{eq2_phi_modify}) separately for each interval because the solution 
must satisfy one of the conditions in (\ref{eq_BC}), namely, $u(0) = l$.
The space $W$ is defined using the weighted functions $\bar \varphi_\pm$, 
and each $U \in W$ is measured using the standard sup-norm
$\| \varphi \|_I$ on an interval $I$, which is defined by 
$\| \varphi \|_I \equiv \sup_{t \in I} |\varphi(t)|$
for a (piecewise) continuous function $\varphi$ on $I$. 
This construction enables us to introduce the mapping $T$ 
(see Lemma~\ref{lemma_estimate_T}). 
Finally, we demonstrate that $T$ is a contraction mapping 
(Lemma~\ref{lemma_Lipschitz_T}).
Hence, the proof of the theorem is completed using the contraction mapping theorem
\cite{MR1814364}.

In Section~\ref{sec_homoclinic}, we establish Theorem~\ref{thm_4}. 
The proof is similar to that of Theorem~\ref{thm_1}, 
again using an iterative scheme applied to (\ref{eq4_phi_modify}). 
However, unlike in the previous case, we require not only the approximate errors 
near $t = 0$ and $t = t_*$, but also an additional error function $\psi$ 
away from these points (see Lemma~\ref{lemma4_linearsol}).

In Section~\ref{sec_periodic}, we establish Theorem~\ref{thm_3}. 
In this setting, the previously mentioned iterative strategy cannot be used.
When the linearized equation in (\ref{eq_u}) is introduced around
the approximate profile, the associated linearized operator is not invertible
(compare (\ref{eq4_detM}) with (\ref{eq3_detM})).
Such a situation also arises in reaction-diffusion systems such 
as in the Allen--Cahn equation \cite{allen1979microscopic}, which 
is a prototypical reaction-diffusion model that describes the phase separation 
and interface motion in multi-component materials. 
Formulated as a gradient flow of the double-well potential,
it exhibits energy dissipation and generates diffuse interfaces 
whose sharp interface limit is governed by mean curvature flow \cite{evans1992phase}. 
Its steady state, traveling fronts, and interface dynamics have been 
extensively studied (\cite{carr1989metastable, fife1977approach}).
Constraint problems were also considered in the Allen--Cahn equation 
\cite{chen2010mass, golovaty1997volume}.
In particular, the relationship between the domain size in which 
the stable states reside and the structure of the nonlinear term 
was investigated by \cite{bronsard1997volume}, 
assuming the existence of a transition layer connecting the two stable equilibria.
When their results are applied to a one-dimensional case, we find the existence 
of stable periodic solutions of the Allen--Cahn equation
whose spatial periods are determined by the prescribed conserved quantity.
Motivated by \cite{bronsard1997volume}, we take the constraint (\ref{eq_const}) 
into account and introduce the modified equation (\ref{eq_u_periodic}).
As discussed at the beginning of Section~\ref{sec_periodic}, 
the periodic solution to (\ref{eq_u_periodic}) becomes
that of (\ref{eq_u}) and (\ref{eq_const}).
We will construct a periodic solution by applying a perturbation method 
to (\ref{eq_u_periodic}).

We conclude this section with a few remarks on the limiting case. 
Although the step function $V(x) = V_0 H (x-l) + V_1$ itself 
is not continuous at $x = l$, 
it is possible to construct homoclinic and periodic solutions of (\ref{eq_u}) 
based on $(u_i, c_0)$ and $(u_d, c_0)$. 
Specifically, the following two theorems remain valid. 
As their proofs follow the same arguments as those in Sections~\ref{sec_homoclinic} 
and \ref{sec_periodic}, we omit these details.


\begin{theorem}\label{thm_H1}
 For any $\e > 0$, let $R > 0$ be sufficiently large. 
 If $t_* > 0$ is sufficiently large, then (\ref{eq_u_infty}) has a homoclinic solution 
 $(u, c)$ for some $\eta > 0$ that satisfies $u(0) = u(t_*) = l$.
 Moreover, there exists $C > 0$ independent of $\e, R$, and $t_*$ 
 such that $|c - c_*| \le C \e$, $|\eta - \eta_*| \le C \e$, 
 $|u (t) - u_i (t; \eta_*)| \le C \e$ 
 (resp. $|u (t) - u_d (t; \eta_*)| \le C \e$) in $t \le t_* - 2 R$,
 and $|u (t) - u_d (t - t_*; \eta)| \le C \e$ 
 (resp. $|u (t) - u_i (t - t_*; \eta)| \le C \e$) in $t_* - 2 R < t$.
\end{theorem}


\begin{theorem}\label{thm_H2}
 Assume that $t_* \neq N / 2$ satisfies (A$_2$).
 For any $\e > 0$, let $R > 0$ be sufficiently large. 
 If $N$ is sufficiently large, then (\ref{eq_u_infty}) has a periodic solution 
 $(u, c)$ for some $\eta > 0$ satisfying $u(0) = u(t_*) = l$ and (\ref{eq_const}). 
 Moreover, there exists $C > 0$ independent of $\e, N$, and $R$ 
 such that $|c - c_*| \le C \e$, $|\eta - \eta_*| \le C \e$,
 $|u (t) - u_i (t; \eta)| \le C \e$ in $- 2 R \le t \le t_* - 2 R$,
 and $|u (t) - u_d (t - t_*; \eta)| \le C \e$ in $t_* - 2 R \le t \le N - 2R$.
\end{theorem}


\section{Heteroclinic solution}\label{sec_Heteroclinic}

In this section, we will construct a heteroclinic solution of (\ref{eq_u}) 
with (\ref{eq_BC}).
The same argument as in this section enables us to prove the existence of 
a heteroclinic solution of (\ref{eq_u}) with (\ref{eq_BC_2}).
Let $0 < \eta < V_0 / a$ be arbitrary.
Set $c \to c_0 + c$.
We derive the lowest-order term of $c$ close to $0$ in (\ref{eq_u_infty}).
Substituting $u (t) = u_i (t) + \phi (t)$ into (\ref{eq_u_infty}), we formally obtain
\begin{equation}\label{eq2_lowest_ceta}
 \phi'' (t) + \k \phi' (t) + \dfrac{2 c}{c_0} u_i'' (t) 
  + \dfrac{a c}{c_0^2} u_i' (t) = 0,
  \quad t \in I_1 \cup I_2,
\end{equation}
where $I_1 \equiv (- \infty, 0)$ and $I_2 \equiv (0, \infty)$.
Keeping (\ref{eq_BC}) in mind, we solve (\ref{eq2_lowest_ceta}) under 
\begin{equation}\label{eq2_StrumLiouville_BC}
 \begin{aligned}
  & \phi \in C^1 (\overline{I_1}) \cap C^1 (\overline{I_2}),
  \quad \phi (0) = 0, \\
  & \lim_{t \to - \infty} |\phi' (t)| \bar \varphi_- (t) < \infty,
  \quad \lim_{t \to \infty} |\phi (t)| \bar \varphi_+ (t) < \infty,  
 \end{aligned}
\end{equation}
where $\bar \varphi_\pm (t) = e^{\pm \k t / 2}$ for $\k = a / c_0$.
It is easy to see that a unique solution $\phi = c \phi_0$ exists, 
where $\phi_0$ is explicitly given by 
\begin{equation}\label{eq2_phi0}
 \phi_0 (t) = 
  \left\{
   \begin{aligned}
    & \dfrac{\eta a}{c_0^2},
    & & t < -1, \\
    & - \dfrac{V_0}{c_0^2} t
    + \dfrac{1}{c_0^2} ((1 + t) e^{- \k t} - 1) (V_0 - a \eta), 
    & & -1 \le t < 0, \\
    & - \dfrac{\eta a}{c_0^2} t e^{- \k t}, 
    & & t \ge 0.
   \end{aligned}
  \right.
\end{equation}
Moreover, it holds that 
\begin{equation}\label{eq2_definition_zeta0}
 \phi_0' (- 1 + 0) - \phi_0' (- 1 - 0) = \zeta_0
  \equiv \dfrac{a V_0}{c_0^3} e^{- \k} > 0,
\end{equation}
which implies that $(\phi_0, \zeta_0)$ can be thought of as a solution of 
(\ref{eq2_lowest_ceta}) and (\ref{eq2_definition_zeta0}).
When we consider (\ref{eq2_lowest_ceta}) and (\ref{eq2_definition_zeta0}) 
under (\ref{eq2_StrumLiouville_BC}) by replacing $u_i$ into $u_d$,
the solution can be given by $- (\phi_0, \zeta_0)$ 
by the same argument as in (\ref{eq_H_symmetry}).

Let $\tilde u (t) = u_i (t) + c \phi_0 (t) + \phi (t)$.
Substituting $\tilde u (t)$ into (\ref{eq_u}), we have 
\begin{equation}\label{eq2_phi}
 \phi'' (t) + \k \phi' (t) = F (t; U),
  \quad \quad t \in I_1 \cup I_2,
\end{equation}
where $U = (\phi, c)$, 
\[
F (t; U) \equiv F_1 (t; U) + \dfrac{a}{(c_0 + c)^2} F_2 (t; U),
\]
\[
\begin{aligned}
 F_1 (t; U) 
 & \equiv 
 - a \left( \dfrac{1}{c_0 + c} - \dfrac{1}{c_0} \right) (c \phi_0' (t) + \phi' (t)) \\
 & - \left( \dfrac{c^2}{(c_0 + c)^2} 
 + \dfrac{2 c_0 c}{(c_0 + c)^2} - \dfrac{2 c}{c_0} \right) u_i'' (t)
 - a c \left( \dfrac{1}{(c_0 + c)^2} - \dfrac{1}{c_0^2} \right) u_i' (t) 
\end{aligned}
\]
and 
\[
F_2 (t; U) \equiv V (\tilde u (t + 1) ) - V (\tilde u (t))
- V_0 (H (u_i (t + 1) - l) - H (u_i (t) - l)).
\]
Since $\tilde u$ must be of class $C^1$, we also consider 
the matching condition at $t = 0$, given by 
\begin{equation}\label{eq2_C1matching_phi}
 \phi' (+0) - \phi' (-0) = - c \zeta_0.
\end{equation}

We will show the existence of $U = (\phi, c)$ satisfying (\ref{eq2_phi})
and (\ref{eq2_C1matching_phi}) in a functional space $W$, given by 
\[
W \equiv \{ U = (\phi, c) \in C (\r) \times \r \mid 
\phi \in C^1 (\overline{I_1}) \cap C^1 (\overline{I_2}),
\ \phi (0) = 0 \},
\]
equipped with the norm $\| U \|_W \equiv \| \phi \| + |c|$,
where the norm $\| \cdot \|$ is given by 
$\| \phi \| \equiv \| \phi' \bar \varphi_- \|_{I_1}
+ \| \phi \bar \varphi_+ \|_{I_2}
+ \| \phi' \bar \varphi_+ \|_{I_2}$.
Actually, we can also estimate $\phi$ on $I_1$ 
by using $\| \phi' \bar \varphi_- \|_{I_1}$.
Thanks to $\phi (0) = 0$, 
\begin{equation}\label{ea2_poincare_ineq}
 \phi (t) = - \int_t^0 \phi' (s) ds
\end{equation}
for $t < 0$, which implies that there is $C > 0$ such that 
\begin{equation}\label{eq2_phi_sup}
 \| \phi \|_{I_1} \le C \| \phi' \bar \varphi_- \|_{I_1}.
\end{equation}
In particular, it follows from (\ref{ea2_poincare_ineq}) 
that $\phi (-\infty)$ exists.
Clearly, $W$ is Banach space.
Moreover, we define a closed ball in $W$ by 
$W_0 \equiv \{ U \in W \mid \| U \|_W \le \b^{-1/2} \}$.
We will find a solution $U$ in $W_0$.

In order to obtain the solution of (\ref{eq2_phi}), 
let us formulate our problem as follows.
For any $U \in W_0$, we first find a solution 
$\bar U = (\bar \phi, \bar c) \in W_0$
satisfying a problem as the modification of (\ref{eq2_phi}) 
and (\ref{eq2_C1matching_phi}), given by 
\begin{equation}\label{eq2_phi_modify}
 \bar \phi'' (t) + \k \bar \phi' (t) = F (t; U),
  \quad t \in I_1 \cup I_2
\end{equation}
and 
\begin{equation}\label{eq2_C1matching_phi_modify}
 \bar \phi (+ 0) - \bar \phi (- 0) = - \bar c \zeta_0
\end{equation}
(see Lemma~\ref{lemma_linearsol}).
Using the solution $\bar U$, we define $T$ mapping from $W_0$ into $W_0$
by $T (U) = \bar U$ (see Lemma~\ref{lemma_estimate_T}).
Secondly, we will show that there exists a fixed point $U$ of $T$
by applying the contraction mapping theorem to $T$ 
(see Lemma~\ref{lemma_Lipschitz_T}),
which completes the proof of Theorem~\ref{thm_1}.
Throughout this section, we denote general positive constants 
independent of $\b$ and $U$ by $C$.

To begin with, we study a linear equation associated with the left-hand side 
of (\ref{eq2_phi_modify}).
Define $W' \equiv C ((-\infty, -1]) \cap C ([-1, 0]) \cap C ([0, \infty))$,
equipped with the norm $\| \cdot \|_{W_\infty'}$ given by 
$\| f \|_{W'} \equiv \| f \bar \varphi_- \|_{I_1} + \| f \bar \varphi_+ \|_{I_2}$.
We find a solution $U = (\phi, c) \in W$ for $f \in W'$ such that $U$ satisfies
\begin{equation}\label{eq2_L}
 \left\{
  \begin{aligned}
   & \phi'' (t) + \k \phi' (t) = f,
   \quad t \in I_1 \cup I_2, \\
   & \phi (+ 0) - \phi (- 0) = - c \zeta_0,
  \end{aligned}
 \right.
\end{equation}
where we omit the bars above the letters in the expressions for simplicity.


\begin{lemma}\label{lemma_linearsol}
 Let $f \in W'$.
 Then there exists a unique solution $U \in W$ of (\ref{eq2_L}).
 Moreover, there exists $C > 0$ independent of $\b$ such that 
 $\| U \|_W \le C \| f \|_{W'}$.
\end{lemma}


\begin{proof}
 Let $\varphi_\pm (t) = e^{\pm \k t}$.
 Note that $\varphi_-$ is a fundamental solution of (\ref{eq2_L})
 and $\varphi_+ = 1 / \varphi_-$.
 Denote the solution of the first equation of (\ref{eq2_L}) 
 on $I_j$ by $\phi_j$ for $j = 1, 2$.
 The solution $\phi_1$ on $I_1$ is given by 
 \[
 \phi_1 (t) = A (\varphi_- (t) - 1)
 + \varphi_- (t) \int_t^0 \dfrac{1}{\k} \varphi_+ (s) f (s) ds
 - \int_t^0 \dfrac{1}{\k} f (s) ds
 \]
 for some $A \in \r$.
 Differentiating the both sides above, we have 
 \[
 \phi_1' (t) = - \k A \varphi_- (t)
 - \varphi_- (t) \int_t^0 \varphi_+ (s) f (s) ds.
 \]
 From the third condition of (\ref{eq2_StrumLiouville_BC}),
 $A$ is determined by 
 \[
 A = - \int_{-\infty}^0 \dfrac{1}{\k} \varphi_+ (s) f (s) ds.
 \]
 Then we have
 \begin{equation}\label{eq2_phi1}
  \phi_1 (t) = \int_{-\infty}^0 \dfrac{1}{\k} \varphi_+ (s) f (s) d s
   - \varphi_- (t) \int_{-\infty}^t \dfrac{1}{\k} \varphi_+ (s) f (s) ds
   - \int_t^0 \dfrac{1}{\k} f (s) ds.
 \end{equation}
 From (\ref{eq2_phi1}), we have
 \begin{equation}\label{eq2_dphi1}
  \phi_1' (t) = \varphi_- (t) \int_{-\infty}^t \varphi_+ (s) f (s) ds.
 \end{equation}
 Similarly, we solve (\ref{eq2_L}) on $I_2$ under $\phi_2 (0) = 0$ 
 and $\lim_{t \to \infty} |\phi_2 (t)| \bar \varphi_+ (t) < \infty$.
 Then $\phi_2$ is given by 
 \begin{equation}\label{eq2_phi2}
  \phi_2 (t) = \varphi_- (t)
   \int_0^{\infty} \dfrac{1}{\k} f (s) ds
   - \varphi_- (t) \int_0^t \dfrac{1}{\k} \varphi_+ (s) f (s) ds
   - \int_t^{\infty} \dfrac{1}{\k} f (s) ds.
 \end{equation}
 Let $\phi$ be a function defined by 
 $\phi = \phi_j$ on $\overline{I_j}$ for $j = 1, 2$.
 Finally, from the second equation of (\ref{eq2_L}), we obtain
 \begin{equation}\label{eq2_c}
  c \zeta_0 = - (\phi_2' (0) - \phi_1' (0))
   = \int_0^{\infty} f (s) d s + \int_{-\infty}^0 \varphi_+ (s) f (s) d s,
 \end{equation}
 which determines $c$.

 We estimate $U$.
 It follows from (\ref{eq2_dphi1}) that 
 $\| \bar \varphi_- \phi_1' \|_{I_1} \le C \| f \|_{W'}$ on $I_1$.
 Similarly, by using (\ref{eq2_phi2}) and (\ref{eq2_c}), 
 we complete the proof of the lemma.
\end{proof}


Next, we study $F (t; U)$.


\begin{lemma}\label{lemma_estimate_F}
 For $U \in W$, $F (t; U) \in W'$.
\end{lemma}


\begin{proof}
 It is trivial that $F_1 (t; U) \in W'$.
 On the other hand, we have
 \begin{equation}\label{eq2_F2_0}
  F_2 (t; U) = V (\tilde u (t + 1)) - V (\tilde u (t))
   = \left( \int_0^1 V' (\bar x) d \theta \right)
   \left( \int_t^{t+1} \tilde u' (\tau) d \tau \right),
 \end{equation}
 where $\bar x \equiv \theta \tilde u (t + 1) + (1 - \theta) \tilde u (t)$.
 Since $V'$ is bounded, we easily see $F_2 (t; U) \in W'$ so that $F (t; U) \in W'$.
\end{proof}


Lemmas~\ref{lemma_linearsol} and \ref{lemma_estimate_F} 
lead to the existence of a unique solution 
$\bar U \in W$ of (\ref{eq2_phi_modify}) and (\ref{eq2_C1matching_phi_modify}).
Then a mapping $T$ is defined by $\bar U = T (U)$.


\begin{lemma}\label{lemma_estimate_T}
 For $U \in W_0$, $T (U) \in W_0$.
\end{lemma}


\begin{proof}
 Let $f_k (t) = F_k (t; U)$ for $k = 1, 2$.
 From Lemmas~\ref{lemma_linearsol} and \ref{lemma_estimate_F} 
 we have a unique solution $\bar U_k = (\bar \phi_k, \bar c_k) \in W$ 
 of (\ref{eq2_L}) for $f (t) = f_k (t)$.
 It is easy to see $\bar U_1 \in W_0$ from Lemma~\ref{lemma_linearsol} because
 $\| \phi' \|_{W'} \le \b^{-1/2}$.

 We show $\bar U_2 \in W_0$.
 Let $\bar x$ be given in (\ref{eq2_F2_0}).
 We will pay more careful attention to the estimate of $\bar U_2$ than $\bar U_1$.
 Let $\bar \phi_{2j} (t) = \bar \phi_2 (t) |_{I_j}$ for $j = 1, 2$.
 We first consider $\bar \phi_{21}$.
 Since $u_i$ increases monotonically, we see
 \begin{equation}\label{eq2_F2_1}
  u_i (t) \le u_i (- \dfrac12) < l,
   \quad |c \phi_0 (t) + \phi (t)| \le C \b^{-1/2}
 \end{equation}
 in $t \le - 1 / 2$ by (\ref{ea2_poincare_ineq}).
 Then, there is $\rho > 0$ independent of $\b$ such that 
 $\bar x \le l - \rho$ in $t \le - 3 / 2$.
 Then we see that $|V' (\bar x)| \le C \b e^{- \a \b \rho}$.
 It follows from (\ref{eq2_ui_explicit}), (\ref{eq2_phi0}), 
 and $\phi \in W_0$ that 
 \[
 \bar \varphi_- (t) \left| \int_t^{t+1} \tilde u' (\tau) d \tau \right| \le C.
 \]
 Then we obtain $\bar \varphi_- (t) |f_2 (t)| \le C \b e^{- \a \b \rho}$
 in $t \le - 3 / 2$
 so that $\bar \varphi_- (t) |\bar \phi_{21}' (t)| \le C \b e^{- \a \b \rho}$ 
 in $t \le - 3 / 2$.

 Next, we consider $- 3 / 2 \le t \le -1 - \b^{-3/4}$.
 Since $l = \tilde u (0)$, there exists $- 1 / 2 \le t + 1 < \bar t < 0$ such that 
 \[
 \tilde u (t + 1) - l = \tilde u' (\bar t) (t + 1)
 \le - \dfrac12 \b^{-3/4} \min_{- 1/2 \le t \le 0} u_i' (t) < 0
 \]
 by the mean value theorem.
 This and (\ref{eq2_F2_1}) imply that 
 $\bar x \le l - \b^{-3/4} \min_{- 1/2 \le t \le 0} u_i' (t) / 2$.
 By setting $\bar \a = \a \min_{- 1/2 \le t \le 0} u_i' (t) / 2$, we have
 $|V' (\bar x)| \le C \b e^{- \bar \a \b^{1/4}}$.
 Then $|f_2 (t)| \le C \b^{-3/4}$.
 On the other hand, $|f_2 (t)| \le C$ 
 in $-1 - \b^{-3/4} \le t \le -1$ because $V$ is bounded.
 Hence we obtain
 \[
 \int_{-3/2}^{-1} \varphi_+ (s) |f_2 (s)| ds
 = \int_{-3/2}^{-1 - \b^{-3/4}} \varphi_+ (s) |f_2 (s)| ds
 + \int_{-1 - \b^{-3/4}}^{-1} \varphi_+ (s) |f_2 (s)| ds \le C \b^{-3/4} 
 \]
 so that $|\bar \phi_{21}' (t)| \le C \b^{-3/4}$ in $-3/2 \le t \le -1$.

 We next consider $-1 \le t \le 0$.
 Then $f_2$ is rewritten into $f_2 (t) = f_{21} (t) - f_{22} (t)$, 
 where $f_{21} (t) = V (\tilde u (t + 1)) - V_0 - V_1$ 
 and $f_{22} (t) = V (\tilde u (t)) - V_1$.
 In $- 1 \le t \le -1 + \b^{-3/4}$, $|f_{21} (t)| \le C$.
 Next, we consider $-1 + \b^{-3/4} \le t \le 0$.
 By the mean value theorem, there exists $\b^{-3/4} < \bar t < 1$ such that 
 \[
 \tilde u (t + 1) - l = \tilde u' (\bar t) (t + 1)
 \ge \dfrac12 \b^{-3/4} \min_{0 \le t \le 1} u_i' (t) > 0. 
 \]
 Setting $\bar \a = \a \min_{0 \le t \le 1} u_i' (t) / 2$, we have
 $|f_{21} (t)| \le C e^{- \bar \a \b^{1/4}}$ from the assumption (A$_1$).
 On the other hand, we can estimate $f_{22}$ 
 by the same arguments as for $f_{21}$ in $t \le -1$.
 Hence we obtain $|\bar \phi_{21}' (t)| \le C \b^{-3/4}$.

 In the same way as in $-1 \le t \le 0$, we easily see
 $|\bar \phi_{22} (t)| \le C \b^{-3/4}$
 and $|\bar \phi_{22}' (t)| \le C \b^{-3/4}$ in $0 \le t \le 1$.
 Finally we consider $t \ge 1$.
 We see 
 \begin{equation}\label{eq2_F2_3}
  u_i (t) \ge u_i (1) > l,
   \quad |c \phi_0 (t) + \phi (t)| \le C \b^{-1/2}.
 \end{equation}
 Then, there is $\rho > 0$ independent of $\b$ such that 
 $\bar x \ge l + \rho$ in $t \ge 1$.
 By the same arguments as in $t \le -3/2$, we have 
 $\bar \varphi_+ (t) |f_2 (t)| \le C \b e^{- \a \b \rho}$ in $t \ge 1$
 so that $\bar \varphi_+ (t) |\bar \phi_{22} (t)| \le C \b^{-3/4}$ 
 and $\bar \varphi_+ (t) |\bar \phi_{22}' (t)| \le C \b^{-3/4}$ in $t \ge 1$.

 From (\ref{eq2_c}), we obtain $|\bar c| \le C \b^{-3/4}$.
 Since $\b$ is sufficiently small, we see $\bar U \in W_0$.
\end{proof}


Next, we prove that If $\b$ is sufficiently small, $T$ is a contraction, 
which implies that $T$ has a fixed point $U \in W_0$, and leads to Theorem~\ref{thm_1}.


\begin{lemma}\label{lemma_Lipschitz_T}
 If $\b$ is sufficiently small, then there exists $C > 0$ independent of $\b$
 such that 
 \[
 \| T (U_p) - T (U_q) \|_W \le C \b^{-1/2} \| U_p - U_q \|_W
 \]
 in $U_p, U_q \in W_0$.
\end{lemma}


\begin{proof}
 Let $U_r = (\phi_r, c_r)$, $\bar U_r = (\bar \phi_r, \bar c_r) = T (U_r)$
 for $r = p, q$.
 Also, we denote $\tilde u_r (t) \equiv u_i (t) + c_r \phi_0 (t) + \phi_r (t)$.
 for simplicity.
 From Lemmas~\ref{lemma_linearsol} and \ref{lemma_estimate_F},
 we have a unique solution 
 $\bar U_{rk} = (\bar \phi_{rk}, \bar c_{rk}) \in W$ of (\ref{eq2_L})
 for $f (t) = f_{rk} (t) \equiv F_k (t; U_r)$ for $r = p, q$ and $k = 1, 2$.
 It readily follows from (\ref{eq2_dphi1})--(\ref{eq2_c}) that 
 $\| \bar U_{p1} - \bar U_{q1} \|_W \le C \b^{-1/2} \| U_p - U_q \|_W$.

 Next, we consider $f_{p2} - f_{q2}$.
 Let $\bar \phi_{r2j}$ be functions defined by (\ref{eq2_phi1}) and (\ref{eq2_phi2})
 for $f (t) = f_{r2} (t)$ on $I_j$ for $j = 1, 2$ and $r = p, q$.
 We first consider $t \le -3/2$.
 As seen in (\ref{eq2_F2_0}), we have
 \[
 F_2 (t; U_r) = \left( \int_0^1 V' (\bar x_r) d \theta \right)
 \left( \int_t^{t+1} \tilde u_r' (\tau) d \tau \right),
 \]
 where $\bar x_r \equiv \theta \tilde u_r (t + 1) + (1 - \theta) \tilde u_r (t)$.
 Moreover, there is $\bar x$ such that 
 \begin{equation}\label{eq2_d2V_meanvalue}
  V' (\bar x_p) - V' (\bar x_q) = V'' (\bar x) (\bar x_p - \bar x_q)
 \end{equation}
 by the mean value theorem.
 By the same inequality as in (\ref{eq2_F2_1}) and from the assumption (A$_1$), 
 there is $\rho > 0$ independent of $\b$ such that $\bar x \le l - \rho$,
 and then $|V'' (\bar x)| \le C \b^2 e^{-\a \b \rho}$.
 Then we obtain
 \begin{equation}\label{eq2_phip21phiq21}
  \bar \varphi_- (t) |\bar \phi_{p21}' (t) - \bar \phi_{q21}' (t)|
   \le C \b^{-1/2} \| U_p - U_q \|_W 
 \end{equation}
 in $t \le -3/2$.

 Next, we consider $- 3 / 2 \le t \le -1$.
 By the same argument as in (\ref{eq2_F2_0}), we see 
 \[
 f_{2p} (t) - f_{2q} (t)
 = \left( \int_0^1 V' (\bar x_1) d \theta \right)
 (\tilde u_p (t + 1) - \tilde u_q (t + 1))
 - \left( \int_0^1 V' (\bar x_2) d \theta \right)
 (\tilde u_p (t) - \tilde u_q (t)), 
 \]
 where 
 $\bar x_1 = \theta_1 \tilde u_p (t + 1) + (1 - \theta_1) \tilde u_q (t + 1)$
 and $\bar x_2 = \theta_2 \tilde u_p (t) + (1 - \theta_2) \tilde u_q (t)$.
 By the same argument as in (\ref{eq2_F2_1}), 
 there is $\rho > 0$ independent of $\b$ such that $\bar x_2 \le l - \rho$, 
 which implies $|V' (\bar x_2)| \le C \b e^{- \a \b \rho}$ in $- 3 / 2 \le t \le -1$.
 Moreover, there is $\bar \a > 0$ 
 independent of $\b$ such that $\bar x_1 \le l - \bar \a \b^{-3/4}$, which implies
 $|V' (\bar x_1)| \le C \b e^{- \bar \a \b^{1/4}}$
 in $- 3 / 2 \le t \le -1 - \b^{-3/4}$.
 On the other hand, $|V' (\bar x_1)| \le C \b$ and 
 \[
 | \tilde u_p (t + 1) - \tilde u_q (t + 1) |
 = \left| \int_{t+1}^0 ( \tilde u_p' (\tau) - \tilde u_q' (\tau) ) d \tau \right|
 \le C \b^{-3/4} \| U_p - U_q \|_W 
 \]
 in $-1 - \b^{-3/4} \le t \le -1$.
 Hence we obtain
 \[
 \begin{aligned}
  & \int_{-3/2}^{-1} \varphi_+ (s) |f_{2p} (s) - f_{2q} (s)| ds \\
  & = \int_{-3/2}^{-1 - \b^{-3/4}} \varphi_+ (s) |f_{2p} (s) - f_{2q} (s)| ds
  + \int_{-1 - \b^{-3/4}}^{-1} \varphi_+ (s) |f_{2p} (s) - f_{2q} (s)| ds \\
  & \le C \b^{-1/2} \| U_p - U_q \|_W   
 \end{aligned} 
 \]
 so that (\ref{eq2_phip21phiq21}) holds true in $-3/2 \le t \le -1$.

 By the same arguments as in $t \le - 1$ and the proof of the previous lemma, 
 we can estimate $f_{p2} (t) - f_{q2} (t)$ in $-1 \le t \le 0$ and $t \ge 0$.
 So we omit the details of the proofs.
 As a result, we have 
 $\| \bar U_{p2} - \bar U_{q2} \|_W \le C \b^{-1/2} \| U_p - U_q \|_W$.
 Also, it follows from (\ref{eq2_c}) and the estimates above that 
 $|\bar c_p - \bar c_q| \le C \b^{-1/2} \| U_p - U_q \|_W$.
 Therefore we complete the proof of the lemma.
\end{proof}


\section{Homoclinic solution}
\label{sec_homoclinic}

In this section, we will construct a homoclinic solution of (\ref{eq_u}) 
satisfying (\ref{eq_formalsol}).
The same argument as in this section enables us to prove the existence of 
a homoclinic solution of (\ref{eq_u}) satisfying (\ref{eq_formalsol2}).
Set $c_* = c_0 (\eta_*)$ and $c \to c_* + c$.
We consider (\ref{eq_u}) in $(-\infty, 0) \cup (0, t_*) \cup (t_*, \infty)$
($\equiv I_1 \cup I_2 \cup I_3$) under $u (0) = u (t_*) = l$.
We suppose that the solution satisfies the matching conditions at $t = 0, t_*$ as
\begin{equation}\label{eq4_C1_matching}
 \left\{
  \begin{aligned}
   & u' (+ 0) = u' (- 0), \\
   & u' (t_* + 0) = u' (t_* - 0).
  \end{aligned}
  \right.
\end{equation}

Define $u_2 = l + \eta_*$,
where $\eta_*$ is defined in Section~\ref{sec_Introduction}.
Let $u_0 (t; \eta)$ be a function defined by 
\[
\begin{aligned}
 u_0 (t; \eta) 
 & = u_i (t; \eta_*) (1 - \xi_R (t - R)) + u_2 \xi_R (t - R),
 & & - \infty < t \le t_* - 2 R, \\
 u_0 (t + t_*; \eta)
 & = u_d (t; \eta_* + \eta) \xi_R (t + 2 R) + u_2 (1 - \xi_R (t + 2 R)),
 & & - 2 R \le t < \infty, 
\end{aligned}
\]
where $R$ is assumed to be sufficiently large and independent of $\b$,
and will be determined later.
Here $\xi_R (t) = \xi (t / R)$, where 
$\xi (x) \in C^\infty (\r)$ is a cut-off function satisfying 
\[
\xi (x) =
\left\{
\begin{aligned}
 & 0, 
 & & x \le 0, \\
 & 1, 
 & & x \ge 1.
\end{aligned}
\right.
\]
It is clear that $u_0$ is $C^1$ and twice differentiable in $t \in \r$.
We note that $u_0 (0) = u_0 (t_*) = l$.
Throughout this section, we denote general positive constants 
independent of $\b, R$ and $t_*$ by $C$.

We next derive error terms between $u$ and $u_0$ in (\ref{eq_u})
and the lowest-order term of $c$ and $\eta$ close to $0$.
Let $\bar \varphi_\pm (t) = e^{\pm \k t / 2}$ for $\k \equiv a / c_*$, 
which were introduced in Section~\ref{sec_Heteroclinic}.
Substituting $u (t) = u_i (t; \eta_*) + \phi (t)$ in (\ref{eq_u}),
we formally obtain
\[
\phi'' (t) + \k \phi' (t) = - c
\left( \dfrac{2}{c_*} u_i'' (t; \eta_*) + \dfrac{a}{c_*^2} u_i' (t; \eta_*) \right),
\quad t \in (-\infty, 0) \cup (0, \infty).
\]
By solving this equation under (\ref{eq2_StrumLiouville_BC}), 
a unique solution is given by $\phi_{0i} (t) \equiv c \phi_0 (t; \eta_*)$.
Moreover, we denote
\[
\zeta_{0i} \equiv \phi_{0i}' (+0) - \phi_{0i}' (-0) = c \dfrac{a V_0}{c_*^3} e^{-\k}.
\]
Next, we derive error terms between $u$ and $u_0$ around $t = t_*$.
Replacing $t \to t + t_*$ and substituting 
$u (t + t_*) = u_d (t; \eta_* + \eta) + \phi (t)$ in (\ref{eq_u}), 
we formally obtain
\[
\phi'' (t) + \k \phi' (t)
= - ( c - c_\eta \eta )
\left( \dfrac{2}{c_*} u_d'' (t; \eta_*) + \dfrac{a}{c_*^2} u_d' (t; \eta_*) \right),
\quad t \in (-\infty, 0) \cup (0, \infty),
\]
where $c_\eta$ is defined by 
\[
c_\eta \equiv \dfrac{\pa c_0}{\pa \eta} (\eta_*) = - \dfrac{c_*^2}{V_0} e^\k.
\]
By solving the equation above under (\ref{eq2_StrumLiouville_BC}),
a unique solution can be given by 
$\phi_{0d} (t) \equiv - ( c - c_\eta \eta ) \phi_0 (t; \eta_*)$
because of $u_d (t) = 2 l - u_i (t)$.
Moreover, we denote
\[
\zeta_{0d} \equiv \phi_{0d}' (+0) - \phi_{0d}' (-0)
= - ( c - c_\eta \eta ) \dfrac{a V_0}{c_*^3} e^{-\k}.
\]
In order to emphasize the $(c, \eta)$-dependencies of 
$\phi_{0i}$, $\phi_{0d}$, $\zeta_{0i}$ and $\zeta_{0d}$, 
we may denote $\phi_{0i} (t; c)$, $\phi_{0d} (t; c, \eta)$, 
$\zeta_{0i} (c)$ and $\zeta_{0d} (c, \eta)$.

According to the arguments in Section~\ref{sec_Heteroclinic},
$\phi_{0i} (t)$ is dominant in the error between $u (t)$ and $u_i (t)$ near $t = 0$,
while $\phi_{0d} (t - t_*)$ is so between $u (t)$ and $u_d (t - t_*)$ near $t = t_*$.
In order to evaluate the error away from $t = 0$ and $t = t_*$, 
we introduce a unique solution $\psi$ in a linear problem 
\begin{equation}\label{eq4_outersol}
 \psi'' (t) + \dfrac{a}{c_* + c} \psi' (t) = f_{o1} (t) + f_{o2} (t) 
  \quad (\equiv f_o (t)),
  \quad t \in I_2
\end{equation}
under $\psi (0) = \psi (t_*) = 0$, 
where $f_{oj} (t)$ for $j = 1, 2$ are given by 
\[
\begin{aligned}
 f_{o1} (t) & \equiv - u_0'' (t) - \dfrac{a}{c_* + c} u_0' (t) \\
 & + \dfrac{c_*^2}{(c_* + c)^2} 
 \left( u_i'' (t; \eta_*) + \dfrac{a}{c_*} u_i' (t; \eta_*) 
 \right) (1 - \xi_R (t - R)) \\
 & + \dfrac{c_*^2}{(c_* + c)^2} 
 c \left( \dfrac{2}{c_*} u_i'' (t; \eta_*) 
 + \dfrac{a}{c_*^2} u_i' (t; \eta_*) \right) (1 - \xi_R (t - R)) \\
 & + \dfrac{c_d^2}{(c_* + c)^2} 
 \left( u_d'' (t - t_*; \eta_* + \eta) 
 + \dfrac{a}{c_d} u_d' (t - t_*; \eta_* + \eta) \right) \xi_R (t - t_* + 2 R) \\
 & + \dfrac{c_*^2}{(c_* + c)^2} ( c - c_\eta \eta ) 
 \left( \dfrac{2}{c_*} u_d'' (t - t_*; \eta_*) 
 + \dfrac{a}{c_*^2} u_d' (t - t_*; \eta_*) \right) \xi_R (t - t_* + 2 R)
\end{aligned}
\]
for $c_d \equiv c_0 (\eta_* + \eta)$ and 
\[
\begin{aligned}
 f_{o2} (t) & \equiv 
 \phi_{0i} (t) \xi_R'' (t - R) - \phi_{0d} (t - t_*) \xi_R'' (t - t_* + 2 R) \\
 & + 2 (\phi_{0i}' (t) \xi_R' (t - R) - \phi_{0d}' (t - t_*) \xi_R' (t - t_* + 2 R)) \\
 & + \dfrac{a}{c_* + c} 
 (\phi_{0i} (t) \xi_R' (t - R) - \phi_{0d} (t - t_*) \xi_R' (t - t_* + 2 R)).
\end{aligned}
\]
Let $\tilde \varphi (t) \in C (\r)$ be defined by 
\[
\tilde \varphi (t) = 
\left\{
\begin{aligned}
 & \bar \varphi_- (t), 
 & & t \le 0, \\
 & \min \{ \overline \varphi_+ (t), \overline \varphi_- (t - t_*), t_* \}, 
 & & 0 \le t \le t_*, \\
 & \bar \varphi_+ (t - t_*), 
 & & t_* \le t.
\end{aligned}
\right.
\]
Let $t_0 = (2 \log t_*) / \k$.
Then it holds that 
$\overline \varphi_+ (t_0) = \overline \varphi_- (t_* - t_0) = t_*$.
In order to emphasize the $(c, \eta)$-dependencies of $\psi$ and $f_o$,
we may denote $\psi (t; c, \eta)$ and $f_o (t; c, \eta)$.
We can prove the following lemma in the same way as Lemma~\ref{lemma_psi}
so that we omit the details of the proof.


\begin{lemma}\label{lemma4_psi}
 Let $\e > 0$ be small.
 Suppose that $R$ is sufficiently large depending on $\e$.
 If $t_*$ is sufficiently large depending on $\e, R$, then there is a constant
 $C > 0$ independent of $\e, R$ and $t_*$ such that the solution 
 $\psi \in C^1 (\overline{I_2}) \cap C^2 ([0, t_*-1]) \cap C^2 ([t_*-1, t_*])$ 
 of (\ref{eq4_outersol}) under $\psi (0) = \psi (t_*) = 0$ can be estimated as
 \begin{equation}\label{eq4_psi_estimate}
  |\psi (t)| \le C \e,
   \quad \tilde \varphi (t) |\psi' (t)| \le C \sigma (t) \e
 \end{equation}
 in $t \in I_2$ and $|c| + |\eta| \le \e$, where $\sigma (t)$ is given by 
 \[
 \sigma (t) = 
 \left\{
 \begin{aligned}
  & 1,
  & & t \in [0, t_* - 2 R), \\
  & \tilde \varphi (t),
  & & t \in [t_* - 2 R, t_*].
 \end{aligned}
 \right.
 \]
 In addition, $\psi$ is Lipschitz continuous with respect to $(c, \eta)$, that is,
 \begin{equation}\label{eq4_psi_Lipschitz}
  \begin{aligned}
   |\psi (t; c_p, \eta_p) - \psi (t; c_q, \eta_q)|
   & \le C (|c_p - c_q| + |\eta_p - \eta_q|), \\
   \tilde \varphi (t) |\psi' (t; c_p, \eta_p) - \psi' (t; c_q, \eta_q)|
   & \le C \sigma (t) (|c_p - c_q| + |\eta_p - \eta_q|)
  \end{aligned}
 \end{equation}
 in $t \in I_2$, and $|c_r| + |\eta_r| \le \e$ for $r = p, q$.
\end{lemma}


\begin{proof}
 Let $\hat \varphi (t) = e^{- \hat \k t}$ for $\hat \k = a / (c_* + c)$.
 In the same way as in Section~\ref{sec_Heteroclinic}, 
 $\psi$ can be written explicitly as 
 \begin{equation}\label{eq4_psi1}
  \psi (t) = A_1 (1 - \hat \varphi (t))
   - \int_0^t \dfrac{1}{\hat \k} \hat \varphi (t - s) f_o (s) ds
   + \int_0^t \dfrac{1}{\hat \k} f_o (s) ds,
 \end{equation}
 where the constant $A_1$ is determined 
 by $\psi (t_*) = 0$ because $\hat \varphi (t_*)$ is close to $0$ 
 for sufficiently large $t_*$ and $\hat \k$ is approximated by $\k$ 
 for sufficiently small $\e$.
 From (\ref{eq4_psi1}), $\psi (0) = 0$.

 We see $f_{o1} (t) = 0$ in $t \in (2 R, t_* - 2 R) \cup (t_* - R, t_* -1)$.
 In $(0, R)$, we see that $f_{o1} (t) = - c^2 u_i'' (t; \eta_*) / (c_* + c)^2$, 
 from which $|f_{o1} (t)| \le C \e^2 e^{- 3 \k t / 4}$
 by (\ref{eq2_ui_explicit}).
 Similarly, $|f_{o1} (t)| \le C \e^2 e^{- 3 \k t / 4}$ in $(R, 2 R)$ because
 $|\xi_R' (t)| \le C / R$ and $|\xi_R'' (t)| \le C / R^2$
 if $R$ is sufficiently large depending on $\e$.
 In the same way as in the case of $(0, R)$,
 we can estimate $|f_{o1} (t)| \le C \e^2$ in $(t_* - 1, t_*)$.
 On the other hand, we study $f_{o1} (t)$ in $(t_* - 2 R, t_* - R)$ 
 more carefully than in $(0, 2 R) \cup (t_* - 1, t_*)$.
 Since $u_d (t - t_*) = u_d (-1)$ in $t \in (t_* - 2 R, t_* - R)$, we have 
 \begin{equation}\label{eq4_fo1_int}
  \int_{t_* - 2 R}^t f_{o1} (s) ds 
   = - (u_d (-1; \eta_* + \eta) - u_2) 
   \left( \xi_R' (t - t_* + 2R) + \dfrac{a}{c_* + c} \xi_R (t - t_* + 2R) \right).
 \end{equation}
 Then we obtain
 \[
 \left| \int_0^t f_{o1} (s) ds \right| \le C \e
 \]
 in $t \in I_2$.
 It is easy to estimate the second term in (\ref{eq4_psi1}) as
 \[
 \left| \int_0^t \hat \varphi (t - s) f_{o1} (s) ds \right|
 \le C \e^2 \int_0^t e^{- \k (t-s)/2} e^{- 3 \k s/4} ds 
 \le C \e^2 e^{- \k t /2} 
 \]
 in $0 \le t \le t_* - 2 R$, 
 \[
 \left| \int_0^t \hat \varphi (t - s) f_{o1} (s) ds \right|
 \le C \e^2 e^{- \k t /2} 
 + \dfrac{C \e}{R} \int_{t_* - 2 R}^t e^{- \k (t-s)/2} ds
 \le \dfrac{C \e}{R}
 \]
 in $t_* - 2 R \le t \le t_* - 1$, and 
 \[
 \left| \int_0^t \hat \varphi (t - s) f_{o1} (s) ds \right|
 \le \dfrac{C \e}{R} + C \e^2 \int_{t_* - 1}^t e^{- \k (t-s)/2} ds
 \le C \e^2
 \]
 in $t_* - 1 \le t \le t_*$.

 Next, we estimate $f_{o2} (t)$ in $I_2$.
 We see $f_{o2} (t) = 0$ in $t \in (0, R) \cup (2 R, t_* - 2 R) \cup (t_* - R, t_*)$.
 In $(R, 2 R)$, we easily estimate $|f_{o2} (t)| \le C \e e^{- 3 \k t / 4} / R$
 by (\ref{eq2_phi0}).
 On the other hand, 
 $\phi_{0d} (t - t_*) = \phi_{0d} (-1)$ in $t \in (t_* - 2 R, t_* - R)$ so that 
 \begin{equation}\label{eq4_fo2_int}
  \int_{t_* - 2 R}^t f_{o2} (s) ds 
   = - \phi_{0d} (-1) 
   \left( \xi_R' (t - t_* + 2R) + \dfrac{a}{c_* + c} \xi_R (t - t_* + 2R) \right).
 \end{equation}
 Hence we obtain
 \[
 \left| \int_0^t f_{o2} (s) ds \right| \le C \e
 \]
 in $t \in I_2$, while 
 \[
 \left| \int_0^t \hat \varphi (t - s) f_{o2} (s) ds \right| \le 
 \left\{
 \begin{aligned}
  & C \dfrac{\e}{R} e^{- \k t /2},
  & & 0 \le t \le t_* - 2 R, \\
  & \dfrac{C \e}{R},
  & & t_* - 2 R \le t \le t_*.
 \end{aligned}
 \right.
 \]
 Therefore we have the inequalities of (\ref{eq4_psi_estimate}).

 It is easy to verify (\ref{eq4_psi_Lipschitz}) in a similar way as above
 by replacing $\psi$ and $f_o (t)$ in (\ref{eq4_psi1}) into 
 $\psi_r$ and $f_{or} (t) \equiv f_{o1r} (t) + f_{o2r} (t)$, respectively,
 where $f_{ojr} (t) = f_{oj} (t; c_r, \eta_r)$ for $r = p, q$.
 So we omit the details of the proof.
\end{proof}


We extend the domain of $\psi$ to $\r$ continuously
by setting $\psi (t) \equiv 0$ in $I_1 \cup I_3$.
We define $\nu_1$ and $\nu_2$ by 
$\nu_1 = \psi' (+ 0)$ and $\nu_2 = - \psi' (t_* - 0)$.
In order to emphasize the $(c, \eta)$-dependencies of $\nu_1$ and $\nu_2$,
we may denote $\nu_1 (c, \eta)$ and $\nu_2 (c, \eta)$.
Now we calculate the leading order terms of $\nu_1$ and $\nu_2$.
Expanding $u_d (-1; \eta_* + \eta) - u_2$ as Taylor series with respect to $\eta$, 
we have $u_d (-1; \eta_* + \eta) - u_2 \approx (e^\k - 1) \eta$
by (\ref{eq2_cs}) and (\ref{eq_necessry_cond}), and 
$\phi_{0d} (-1) = - ( c - c_\eta \eta ) a \eta_* / c_*^2$.
Then we see
\[
\left| A_1 + \dfrac{a \eta_*}{c_*^2} c \right| \le C \e^2,
\quad \left| \nu_1 + \dfrac{a^2 \eta_*}{c_*^3} c \right| \le C \e^2,
\quad | \nu_2 | \le C \e^2
\]
in $|c| + |\eta| \le \e$.
Note that $a \eta_* / V_0 = 1 - e^{- \k}$, and then 
$e^\k - 1 + a \eta_* c_\eta / c_*^2 = 0$.
Moreover, it holds that 
\[
\begin{aligned}
 \left| \nu_1 (c_p, \eta_p) + \dfrac{a^2 \eta_*}{c_*^3} c_p 
 - \left( \nu_1 (c_q, \eta_q) + \dfrac{a^2 \eta_*}{c_*^3} c_q \right) \right|
 & \le C \e (|c_p - c_q| + |\eta_p - \eta_q|), \\
 | \nu_2 (c_p, \eta_p) - \nu_2 (c_q, \eta_q) |
 & \le C \e (|c_p - c_q| + |\eta_p - \eta_q|)
\end{aligned}
\]
in $|c_r| + |\eta_r| \le \e$ for $r = p, q$.

Let $W_\infty$ be defined by 
\[
W_\infty \equiv \{ U = (\phi, c, \eta) \in C (\r) \times \r^2 \mid 
\phi \in C^1 (\overline{I_1}) \cap C^1 (\overline{I_2}) \cap C^1 (\overline{I_3}),
\quad \phi (0) = \phi (t_*) = 0 \},
\]
equipped with the norm $\| \cdot \|_{W_\infty}$ defined by 
\[
\| U \|_{W_\infty} \equiv \| \phi \| + |c| + |\eta|, 
\quad \| \phi \| \equiv \| \phi' \tilde \varphi \|_{I_1}
+ \| \phi' \tilde \varphi \|_{I_2}
+ \| \phi \tilde \varphi \|_{I_3} + \| \phi' \tilde \varphi \|_{I_3}.
\]
Clearly, $W_\infty$ is a Banach space.
Actually, we prove $\| \phi \|_\r \le C \| \phi \|$
by the same argument as in (\ref{eq2_phi_sup}).
Let $\e > 0$ be a small parameter independent of $\b, R, t_*$.
We define a closed ball in $W_\infty$ with the radius $\e$ by 
$W_\infty^\e \equiv \{ U \in W_\infty \mid \| U \|_{W_\infty} \le \e \}$.
For $U \in W_\infty$, we give $\tilde u (t) = \tilde u (t; U)$ by 
\[
\begin{aligned}
 \tilde u (t) = u_0 (t) + \phi_{0i} (t) (1 - \xi_R (t - R))
 + \phi_{0d} (t - t_*) \xi_R (t - t_* + 2 R) + \psi (t) + \phi (t).
\end{aligned}
\]
In order to prove Theorem~\ref{thm_4}, it is sufficient to find $U \in W_\infty$ 
such that $\tilde u$ satisfies (\ref{eq_u}) and (\ref{eq4_C1_matching})
for sufficiently small $\e$.
From here until the end of this section, we assume that these parameters are chosen 
in the order $\e, R$ and $(\b, t_*)$.
More precisely, the parameter $\e$ is taken to be sufficiently small, 
and subsequently, for such a fixed $\e$, 
the parameter $R$ is chosen to be sufficiently large.
Finally $\b$ and $t_*$ are assumed to be sufficiently large depending on $\e$ and $R$.
Hereafter, we denote general positive constants 
independent of $\b, \e, R, t_*$ and $U$ by $C$.

In the same way as in Section~\ref{sec_Heteroclinic},
we substitute $u (t) = \tilde u (t)$ into (\ref{eq_u})
and derive a linear equation, given by 
\begin{equation}\label{eq4_phi_modify}
 \bar \phi'' (t) + \k \bar \phi' (t) = F (t; U),
  \quad t \in I_1 \cup I_2 \cup I_3.
\end{equation}
Here the function $F$ in the right-hand side is decomposed into two parts as
\[
F (t; U) \equiv F_1 (t; U) + \dfrac{a}{(c_* + c)^2} F_2 (t; U),
\]
where
\[
\begin{aligned}
 F_1 (t; U) 
 & \equiv - \left( \dfrac{a}{c_* + c} - \dfrac{a}{c_*} \right) \phi' (t) \\
 & - \left( 1 - \dfrac{c_*^2}{(c_* + c)^2} \right) 
 ( \phi_{0i}'' (t) (1 - \xi_R (t - R)) 
 + \phi_{0d}'' (t - t_*) \xi_R (t - t_* + 2 R) ) \\
 &  - \dfrac{a c}{(c_* + c)^2}
 (\phi_{0i}' (t) (1 - \xi_R (t - R)) 
 + \phi_{0d}' (t - t_*) \xi_R (t - t_* + 2 R) )
\end{aligned}
\]
and
\[
\begin{aligned}
 & F_2 (t; U) \\
 & \equiv V (\tilde u (t + 1)) - V (\tilde u (t)) \\
 & - V_0 (H (u_i (t + 1; \eta_*) - l) - H (u_i (t; \eta_*) - l))) (1 - \xi_R (t - R)) \\
 & - V_0 (H (u_d (t - t_* + 1; \eta_* + \eta) - l) 
 - H (u_d (t - t_*; \eta_* + \eta) - l))) \xi_R (t - t_* + 2 R).
\end{aligned}
\]
Similarly, we substitute $\tilde u$ into (\ref{eq4_C1_matching}), obtaining
\begin{equation}\label{eq4_C1_matching_phi}
 \left\{
 \begin{aligned}
  & M_{11} \bar c + M_{12} \bar \eta 
  + \bar \phi' (+0) - \bar \phi' (-0) = G_1 (U), \\
  & M_{21} \bar c + M_{22} \bar \eta 
  + \bar \phi' (t_* + 0) - \bar \phi' (t_* - 0) = G_2 (U),
 \end{aligned}
 \right.
\end{equation}
where $M_{ij}$ and $G_i (U)$ for $i, j = 1, 2$ are given by 
\[
M_{11} \equiv \dfrac{a V_0}{c_*^3} ( 1 - \k ),
\quad M_{12} \equiv 0, 
\quad M_{21} \equiv - \dfrac{a V_0}{c_*^3} e^{-\k},
\quad M_{22} \equiv - \k,
\]
and
\[
G_1 (U) \equiv - \nu_1 - c \dfrac{a^2 \eta_*}{c_*^3},
\quad G_2 (U) \equiv - \nu_2.
\]
It is trivial that a matrix $M \equiv (M_{ij})$ for $i, j = 1, 2$ 
is invertible because its determinant is computed as 
\begin{equation}\label{eq4_detM}
 \det M = - \dfrac{a V_0}{c_*^3} ( 1 - \k ) \k \neq 0.
\end{equation}

We think of (\ref{eq4_phi_modify}) and (\ref{eq4_C1_matching_phi})
as a linear system with respect to $\bar U = (\bar \phi, \bar c, \bar \eta)$.
To begin with, we focus on the left-hand sides of 
(\ref{eq4_phi_modify}) and (\ref{eq4_C1_matching_phi}), 
and study a linear problem given by 
\begin{equation}\label{eq4_L}
 \left\{
  \begin{aligned}
   & \phi'' (t) + \k \phi' (t) = f, 
   \quad t \in I_1 \cup I_2 \cup I_3, \\
   & M_{11} c + M_{12} \eta + \phi' (+0) - \phi' (-0) = g_1, \\
   & M_{21} c + M_{22} \eta + \phi' (t_* + 0) - \phi' (t_* - 0) = g_2,
  \end{aligned}
 \right.
\end{equation}
where we omit the bars above the letters in the expressions for simplicity.
By the same arguments as in the previous section, 
we find a solution $U \in W_\infty$ for $h = (f, g_1, g_2) \in W_\infty'$, 
where $W_\infty'$ is defined by 
\[
W_\infty' \equiv 
(C ((-\infty, -1]) \cap C ([-1, 0]) \cap 
C ([0, t_*-1]) \cap C ([t_*-1, t_*]) \cap C (\overline{I_3})) \times \r^2
\]
with the norm $\| \cdot \|_{W_\infty'}$ given by 
$\| h \|_{W_\infty'} \equiv \| f \tilde \varphi \|_\r + |g_1| + |g_2|$.


\begin{lemma}\label{lemma4_linearsol}
 Let $h \in W_\infty'$.
 Then there exists a unique solution $U \in W_\infty$ of (\ref{eq4_L}) such that 
 $\| U \|_{W_\infty} \le C \| h \|_{W_\infty'}$ for some constant $C > 0$.
\end{lemma}


\begin{proof}
 We show the existence of a unique solution $U$.
 Let $\varphi_\pm (t) = e^{\pm \k t}$, which were introduced in the previous section.
 In the same manner as described in Lemma~\ref{lemma_linearsol}, 
 the solutions on $I_1$ and $I_3$, denoted by $\phi_1$ and $\phi_3$ respectively,
 are given by 
 \begin{equation}\label{eq4_phi1}
  \phi_1 (t) = \int_{-\infty}^0 \dfrac{1}{\k} \varphi_+ (s) f (s) d s
   - \varphi_- (t) \int_{-\infty}^t \dfrac{1}{\k} \varphi_+ (s) f (s) ds
   - \int_t^0 \dfrac{1}{\k} f (s) ds
 \end{equation}
 and 
 \begin{equation}\label{eq4_phi3}
  \phi_3 (t) = \varphi_- (t - t_*) \int_{t_*}^{\infty} \dfrac{1}{\k} f (s) ds
   - \varphi_- (t) \int_{t_*}^t \dfrac{1}{\k} \varphi_+ (s) f (s) ds
   - \int_t^{\infty} \dfrac{1}{\k} f (s) ds.
 \end{equation}
 On the other hand, $\phi_2$ is given by 
 \begin{equation}\label{eq4_phi2}
  \phi_2 (t) = A (1 - \varphi_- (t))
   - \varphi_- (t) \int_0^t \dfrac{1}{\k} \varphi_+ (s) f (s) ds
   + \int_0^t \dfrac{1}{\k} f (s) ds,
 \end{equation}
 where $A$ can be uniquely determined under the condition $\phi_2 (t_*) = 0$.
 We note that $\phi_1 (0) = \phi_2 (0) = 0$ and $\phi_3 (t_*) = 0$.
 Let $\phi (t) = \phi_j (t)$ on $I_j$.
 Since the second and third equations of (\ref{eq4_L}) determine $(c, \eta)$,
 $U$ is a solution of (\ref{eq4_L}).
 Based on the arguments as in the proofs 
 of Lemmas~\ref{lemma_linearsol} and \ref{lemma4_psi}, 
 it is easy to estimate $U$ by (\ref{eq4_detM}).
\end{proof}


We will prove that there exists $\bar U$ 
in (\ref{eq4_phi_modify}) and (\ref{eq4_C1_matching_phi})
such that $\bar U \in W_\infty^\e$.


\begin{lemma}\label{lemma4_estimate_barsol}
 If $U \in W_\infty^\e$, then there exists a unique solution $\bar U \in W_\infty^\e$ 
 in (\ref{eq4_phi_modify}) and (\ref{eq4_C1_matching_phi}).
\end{lemma}


\begin{proof}
 Set $h_1 = (F_1 (t; U), 0, 0)$, $h_2 = (F_2 (t; U), 0, 0)$ 
 and $h_3 = (0, G_1 (U), G_2 (U))$ for $U = (\phi, c, \eta) \in W_\infty^\e$.
 It is easily proved that $h_k \in W_\infty'$ for $k = 1, 2, 3$
 by applying (\ref{eq2_F2_0}) to $F_2 (t; U)$.
 According to Lemma~\ref{lemma4_linearsol}, we have a unique solution 
 $\bar U_k = (\bar \phi_k, \bar c_k, \bar \eta_k) \in W_\infty$ 
 of (\ref{eq4_L}) for $h = h_k$.
 By setting $\bar U = \bar U_1 + \bar U_2 + \bar U_3$, 
 $\bar U$ satisfies (\ref{eq4_phi_modify}) and (\ref{eq4_C1_matching_phi}).

 Let $f (t) = F_1 (t; U)$.
 It is easy to see that $\| f \tilde \varphi \|_\r \le C \e^2$.
 Hence Lemma~\ref{lemma4_linearsol} leads to $\| \bar U_1 \|_{W_\infty} \le C \e^2$.
 Next, we consider $\bar U_2$.
 We define $\phi_j$ by (\ref{eq4_phi1})--(\ref{eq4_phi2}) 
 for $f (t) = F_2 (t; U)$ and $j = 1, 2, 3$.
 By the same argument as in Lemma~\ref{lemma_estimate_T}, we easily obtain 
 $\| \bar \varphi_- \phi_1' \|_{I_1}  \le C \b^{-3/4}$ and 
 $\| \bar \varphi_+ \phi_3 \|_{I_3} 
 + \| \bar \varphi_+ \phi_3' \|_{I_3} \le C \b^{-3/4}$.
 We study the first integral term in (\ref{eq4_phi2}).
 In $0 \le t \le 1$, $f$ is rewritten into $f (t) = f_1 (t) - f_2 (t)$, 
 where $f_1 (t) = V (\tilde u (t + 1)) - V_0 - V_1$
 and $f_2 (t) = V (\tilde u (t)) - V_0 - V_1$.
 By the same arguments as in Lemma~\ref{lemma_estimate_T}, 
 there is $\rho > 0$ independent of $\b, \e, t_*, R$ and $U$ such that
 $|f_1 (t)| \le C e^{- \a \b \rho}$.
 In $0 \le t \le \b^{-3/4}$, $|f_2 (t)| \le C$.
 In $\b^{-3/4} \le t \le 1$, there exists $\b^{-3/4} < \bar t < 1$
 such that 
 \[
 \tilde u (t) - l = \tilde u' (\bar t) t
 \ge \dfrac12 \b^{-3/4} \min_{0 \le t \le 1} u_i' (t) 
 \]
 by the mean value theorem.
 Setting $\bar \a = \a \min_{0 \le t \le 1} u_i' (t) / 2 > 0$, 
 we have $|f_2 (t)| \le C e^{- \bar \a \b^{1/4}}$.
 Hence we can estimate
 \[
 \int_0^1 |f (s)| ds 
 = \int_0^{\b^{-3/4}} |f (s)| ds + \int_{\b^{-3/4}}^1 |f (s)| ds \le C \b^{-3/4}.
 \]

 Next, we consider $1 \le t \le t_* - 3/2$.
 We find
 \begin{equation}\label{eq4_F2_0}
  \begin{aligned}
   f (t) & = V (\tilde u (t + 1)) - V (\tilde u (t))
   = \left( \int_0^1 V' (\bar x) d \theta \right)
   \left( \int_t^{t + 1} \tilde u' (\tau) d \tau \right), \\
   \bar x & = \theta \tilde u (t + 1) + (1 - \theta) \tilde u (t).
  \end{aligned}  
 \end{equation}
 By the same argument as in (\ref{eq2_F2_3}), there exists $\rho > 0$ 
 independent of $\b, \e, R$ and $t_*$ such that $\bar x - l \ge \rho$.
 Then $|V' (\bar x)| \le C \b e^{- \a \b \rho}$.
 On the other hand, we estimate 
 \[
 \tilde \varphi (t) \int_t^{t + 1} |\tilde u' (\tau)| d \tau 
 \le C \sigma (t)
 \]
 by (\ref{eq2_ui_explicit}), (\ref{eq2_phi0}), 
 (\ref{eq4_psi_estimate}), and $U \in W_\infty^\e$,
 where $\sigma (t)$ is given in the statement of Lemma~\ref{lemma4_psi}.
 Then we can estimate 
 \[
 \begin{aligned}
  & \left| \varphi_- (t) \int_0^t \varphi_+ (s) f (s) ds \right| \\
  & \le C \b^{-3/4} e^{- \k t} + C \b e^{- \a \b \rho} 
  \int_1^t e^{- \k (t - s)} 
  \max \{ e^{- \k s / 2}, e^{- \k (t_* - s) / 2}, \dfrac{1}{t_*} \} ds \\
  & \le C \b^{-3/4} 
  \max \{ e^{- \k t / 2}, e^{- \k (t_* - t) / 2}, \dfrac{1}{t_*} \}
 \end{aligned}
 \]
 in $1 \le t \le t_* - 2 R - 1$.
 If $\b$ is sufficiently large depending on $R$, we see
 $\b e^{- \a \b \rho} e^{\k R} \le \b^{-3/4}$ so that 
 \[
 \begin{aligned}
  \left| \varphi_- (t) \int_0^t \varphi_+ (s) f (s) ds \right|
  & \le C \b^{-3/4} e^{- \k (t_* - t) / 2}
  + C \b e^{- \a \b \rho} \int_{t_* - 2 R - 1}^t e^{- \k (t - s)} ds \\
  & \le C \b^{-3/4} e^{- \k (t_* - t) / 2}
 \end{aligned}
 \]
 in $t_* - 2 R - 1 \le t \le t_* - 3/2$.

 Next, we consider $t_* - 3/2 \le t \le t_* - 1$.
 We rewrite $f$ into $f (t) = f_1 (t) - f_2 (t)$, 
 where $f_1 (t) = V (\tilde u (t + 1)) - V_0 - V_1$
 and $f_2 (t) = V (\tilde u (t)) - V_0 - V_1$.
 In $t_* - 3/2 \le t \le t_* - 1 - \b^{-3/4}$, 
 there exists $t_* - 1/2 < \bar t < t_*$ such that 
 \[
 \tilde u (t + 1) - l = - \tilde u' (\bar t) (t_* - (t + 1))
 \ge - \dfrac12 \b^{-3/4} \max_{-1/2 \le t \le 0} u_d' (t) > 0
 \]
 by the mean value theorem.
 Setting $\bar \a = - \a \max_{-1/2 \le t \le 0} u_d' (t) / 2$, we have
 $|f_1 (t)| \le C e^{- \bar \a \b^{1/4}}$.
 On the other hand, there is $\rho > 0$ independent of $\b$ such that
 $\tilde u (t) \ge l + \rho$ by the same argument as above, 
 which implies $|f_2 (t)| \le C e^{- \a \b \rho}$.
 In $t_* - 1 - \b^{-3/4} \le t \le t_* - 1$, 
 it is clear that $|f (t)| \le C$.
 Hence we can estimate
 \[
 \int_{t_* - 3/2}^{t_* - 1} |f (s)| ds 
 = \int_{t_* - 3/2}^{t_* - 1 - \b^{-3/4}} |f (s)| ds 
 + \int_{t_* - 1 - \b^{-3/4}}^{t_* - 1} |f (s)| ds 
 \le C \b^{-3/4}.
 \]
 Hence we see
 \[
 \left| \varphi_- (t) \int_0^t \varphi_+ (s) f (s) ds \right|
 \le C \b^{-3/4}
 \]
 in $t_* - 3/2 \le t \le t_* - 1$ by the same argument as above.
 The inequality above holds true in $t_* - 1 \le t \le t_*$
 by rewriting $f$ into $f (t) = f_1 (t) - f_2 (t)$
 and using the same argument as in $t_* - 3/2 \le t \le t_* - 1$ 
 and $0 \le t \le 1$, 
 where $f_1 (t) = V (\tilde u (t + 1)) - V_1$
 and $f_2 (t) = V (\tilde u (t)) - V_0 - V_1$.

 Since the integral terms in (\ref{eq4_phi2}) can be estimated
 in similar ways to the arguments above, we obtain $|A| \le C \b^{-3/4}$,
 $|\bar c_2| + |\bar \eta_2| \le C \b^{-3/4}$, and 
 $\| \bar U_2 \|_{W_\infty} = \| \bar \phi_2 \| \le C \b^{-3/4}$
 because $\bar \phi_2 (t) = \phi_j (t)$ on $I_j$ for $j = 1, 2, 3$.

 We estimate $\bar U_3$.
 Since $f$ is zero in (\ref{eq4_phi1})--(\ref{eq4_phi2}), $\bar \phi_3 \equiv 0$.
 Since $|G_j (U)| \le C \e^2$ for $j = 1, 2$, 
 we have $\| \bar U_3 \|_{W_\infty} \le C \e^2$
 by Lemma~\ref{lemma4_linearsol}.
 If $\e$ is sufficiently small and $\b$ is sufficiently large, we have 
 $ \| \bar U \|_{W_\infty} \le C \b^{-3/4} + C \e^2 \le \e$.
\end{proof}


By Lemma~\ref{lemma4_estimate_barsol}, we can define a map $T$
from $W_\infty^\e$ to itself by $\bar U = T (U)$.
Next, we prove that if $\b$ is sufficiently small, then $T$ is a contraction.
This implies that $T$ has a fixed point $U \in W_\infty^\e$, 
and hence Theorem~\ref{thm_4} follows.


\begin{lemma}\label{lemma4_Lipschitz_T_homoclinic}
 There exists $C > 0$ independent of $\b, \e, R$ and $t_*$ such that 
 \[
 \| T (U_p) - T (U_q) \|_{W_\infty} \le C \e \| U_p - U_q \|_{W_\infty}
 \]
 in $U_p, U_q \in W_\infty^\e$.
 In particular, $T$ is a contraction mapping and has a fixed point $U \in W_\infty^\e$.
\end{lemma}


\begin{proof}
 Let $U_r = (\phi_r, c_r, \eta_r)$ for $r = p, q$.
 Set $h_{r1} = (F_1 (t; U_r), 0, 0)$, 
 $h_{r2} = (F_2 (t; U_r), 0, 0)$ and $h_{r3} = (0, G_1 (U_r), G_2 (U_r)) \in W_\infty'$.
 Since $h_{rk} \in W_\infty'$ for $k = 1, 2, 3$, we have a unique solution 
 $\bar U_{rk} \in W_\infty^\e$ of (\ref{eq4_L}) for $h_{rk}$ 
 by Lemma~\ref{lemma4_linearsol}.

 We denote $\tilde u_r (t) = \tilde u (t; U_r)$ for simplicity.
 We first claim that 
 \begin{equation}\label{eq4_Lipschitz_ur}
  |\tilde u_p (t) - \tilde u_q (t)| 
   \le C \min \{ |t|, |t - t_*|, 1 \} \| U_p - U_q \|_{W_\infty}.
 \end{equation}
 Since 
 \[
 \tilde u_p (t) - \tilde u_q (t) 
 = \tilde u_p (t) - \tilde u_p (0) - (\tilde u_q (t) - \tilde u_q (0))
 = \int_0^t (\tilde u_p' (s) - \tilde u_q' (s)) d s,
 \]
 (\ref{eq4_Lipschitz_ur}) holds true in $0 \le t \le t_* / 2$
 by (\ref{eq2_ui_explicit}), (\ref{eq2_phi0}), 
 (\ref{eq4_psi_Lipschitz}), and $U_p, U_q \in W_\infty^\e$.
 In other intervals, we easily check (\ref{eq4_Lipschitz_ur})
 by $\tilde u_r (t_*) = l$ for $r = p, q$.
 On the other hand, we also estimate 
 \begin{equation}\label{eq4_Lipschitz_dur}
  \tilde \varphi (t) \int_t^{t + 1} |\tilde u_r' (\tau)| d \tau \le C \sigma (t),
   \quad \tilde \varphi (t) 
   \int_t^{t + 1} |\tilde u_p' (\tau) - \tilde u_q' (\tau)| d \tau
   \le C \sigma (t) \| U_p - U_q \|_{W_\infty}.
 \end{equation}

 It is easy to see 
 \begin{equation}\label{eq4_F1_Lipschitz}
  \| \bar U_{p1} - \bar U_{q1} \|_{W_\infty} \le C \e \| U_p - U_q \|_{W_\infty}
 \end{equation}
 from Lemma~\ref{lemma4_linearsol}.
 Define $f_r (t) = F_2 (t; U_r)$ on $r = p, q$.
 Let $\phi_{rj}$ be defined by (\ref{eq4_phi1})--(\ref{eq4_phi2}) 
 for $f (t) = f_r (t)$ on $I_j$ for $j = 1, 2, 3$.
 We denote the constant $A$ in (\ref{eq4_phi2}) by $A_r$, respectively.
 Then, $\bar \phi_{r2} (t) = \phi_{rj} (t)$ on $I_j$ for $j = 1, 2, 3$.

 By the same argument as in Lemma~\ref{lemma_Lipschitz_T}, we easily obtain 
 \[
 \begin{aligned}
  & \| \bar \varphi_- (\phi_{p1}' - \phi_{q1}') \|_{(-\infty, -3/2)} 
  \le C \b^{-1/2} \| U_p - U_q \|_{W_\infty}, \\
  & \| \bar \varphi_+ (\phi_{p3} - \phi_{q3})  \|_{I_3}
  + \| \bar \varphi_+ (\phi_{p3}' - \phi_{q3}')  \|_{I_3} 
  \le C \b^{-1/2} \| U_p - U_q \|_{W_\infty}.  
 \end{aligned}
 \]
 Consider $-3/2 \le t \le -1$.
 We see 
 \begin{equation}\label{eq4_f2pf2q}
  \begin{aligned}
   f_p (t) - f_q (t)
   & = \left( \int_0^1 V' (\bar x_1) d \theta \right)
   ( \tilde u_p (t + 1) - \tilde u_q (t + 1) ) \\
   & - \left( \int_0^1 V' (\bar x_2) d \theta \right) 
   ( \tilde u_p (t) - \tilde u_q (t) ),
  \end{aligned}  
 \end{equation}
 where $\bar x_1 = \theta \tilde u_p (t + 1) + (1 - \theta) \tilde u_q (t + 1)$
 and $\bar x_2 = \theta \tilde u_p (t) + (1 - \theta) \tilde u_q (t)$.
 By the same arguments as in the previous lemma, there are $\bar \a > 0$ 
 and $\rho > 0$ independent of $\b, \e, R$ and $t_*$ such that
 $|V' (\bar x_2)| \le C \b e^{- \a \b \rho}$
 and 
 \[
 |V' (\bar x_1)| \le
 \left\{
 \begin{aligned}
  & C \b e^{- \bar \a \b^{1/4}},
  & & -3/2 \le t \le -1 - \b^{-3/4}, \\
  & C \b,
  & & - 1 - \b^{-3/4} \le t \le - 1.
 \end{aligned}
 \right.
 \]
 Using (\ref{eq4_Lipschitz_ur}), we obtain
 \begin{equation}\label{eq4_estimate_f2pf2q_1}
  \begin{aligned}
  \int_{-3/2}^{-1} |f_p (s) - f_q (s)| ds 
   & = \left( \int_{-3/2}^{-1-\b^{-3/4}} + \int_{-1-\b^{-3/4}}^{-1} \right) 
   |f_p (s) - f_q (s)| ds \\
   & \le C \b^{-1/2} \| U_p - U_q \|_{W_\infty}.   
  \end{aligned}
 \end{equation}
 In the same way as above, we can obtain a similar estimate 
 to (\ref{eq4_estimate_f2pf2q_1}) in $-1 \le t \le 1$ so that 
 $\| \bar \varphi_- (\phi_{p1}' - \phi_{q1}') \|_{I_1} 
 \le C \b^{-1/2} \| U_p - U_q \|_{W_\infty}$.

 We next consider $1 \le t \le t_* - 3/2$.
 In the same way as in (\ref{eq4_F2_0}), we can define $\bar x_r$
 for $f (t) = f_r (t)$.
 Moreover, there is $\bar x$ such that (\ref{eq2_d2V_meanvalue}) holds 
 by the mean value theorem.
 In the same way as in Lemma~\ref{lemma_estimate_T}, there is $\rho > 0$ 
 independent of $\b, \e, R$ and $t_*$ such that $\bar x \ge l + \rho$,
 and then $|V'' (\bar x)| \le C \b^2 e^{-\a \b \rho}$ from the assumption (A$_1$).
 By the same arguments as in the previous lemma, we have 
 \[
 \tilde \varphi (t)
 \left| \varphi_- (t) \int_0^t \varphi_+ (s) (f_p (s) - f_q (s)) ds \right|
 \le C \b^{-1/2} \| U_p - U_q \|_{W_\infty}
 \]
 in $1 \le t \le t_* - 3/2$.

 Next, we consider $t_* - 3/2 \le t \le t_*$.
 We use (\ref{eq4_f2pf2q}) again.
 By the same arguments as above, there are $\bar \a > 0$ and $\rho > 0$
 independent of $\b, \e, R$ and $t_*$ such that
 \[
 |V' (\bar x_1)| \le
 \left\{
 \begin{aligned}
  & C \b e^{- \bar \a \b^{1/4}}, 
  & & t_* - 3/2 \le t \le t_* - 1 - \b^{-3/4}, 
  \quad t_* - 1 + \b^{-3/4} \le t \le t_*, \\ 
  & C \b, 
  & & t_* - 1 - \b^{-3/4} \le t \le t_* - 1 + \b^{-3/4}
 \end{aligned}
 \right.
 \]
 and 
 \[
 |V' (\bar x_2)| \le 
 \left\{
 \begin{aligned}
  & C \b e^{- \bar \a \b^{1/4}}, 
  & & t_* - 3/2 \le t \le t_* - \b^{-3/4}, \\
  & C \b, 
  & & t_* - \b^{-3/4} \le t \le t_*.
 \end{aligned}
 \right.
 \]
 By the same argument as in (\ref{eq4_estimate_f2pf2q_1}), we obtain
 \[
 \left| \varphi_- (t) \int_0^t \varphi_+ (s) (f_p (s) - f_q (s)) ds \right|
 \le C \b^{-1/2} \| U_p - U_q \|_{W_\infty}. 
 \]

 Since the integral terms in (\ref{eq4_phi1})--(\ref{eq4_phi2}) can be estimated
 in similar ways to the arguments above, we obtain 
 $|A_p - A_q| \le C \b^{-1/2} \| U_p - U_q \|_{W_\infty}$
 and $\| \tilde \varphi (\bar \phi_{p2} - \bar \phi_{q2}) \|_{I_2}
 \le C \b^{-1/2} \| U_p - U_q \|_{W_\infty}$,
 which implies
 \begin{equation}\label{eq4_F2_Lipschitz}
  \| \bar U_{p2} - \bar U_{q2} \|_{W_\infty} 
   \le C \b^{-1/2} \| U_p - U_q \|_{W_\infty}.
 \end{equation}

 It is easy to see that 
 $|G_j (U_p) - G_j (U_q)| \le C \e \| U_p - U_q \|_{W_\infty}$
 for $j = 1, 2$, from which
 \begin{equation}\label{eq4_F3_Lipschitz}
  \| \bar U_{p3} - \bar U_{q3} \|_{W_\infty} \le C \e \| U_p - U_q \|_{W_\infty}.
 \end{equation}
 Since $T (U_r) = \bar U_{r1} + \bar U_{r2} + \bar U_{r3}$,
 we conclude Lemma~\ref{lemma4_Lipschitz_T_homoclinic}
 by (\ref{eq4_F1_Lipschitz}), (\ref{eq4_F2_Lipschitz}) and (\ref{eq4_F3_Lipschitz}).
\end{proof}


According to Lemmas~\ref{lemma4_estimate_barsol} 
and \ref{lemma4_Lipschitz_T_homoclinic},
there exists a fixed point $U$ of $T$ in $W_\infty^\e$.
For this $U = (\phi, c, \eta)$, we defined $u (t) \equiv \tilde u (t; U)$.
Then $(u, c_* + c)$ is a solution of (\ref{eq_u}).
Finally, we prove that it is a homoclinic solution, that is, 
$u (-\infty) = u (\infty)$.


\begin{proof}
 We integrate the both sides of (\ref{eq_u}) on $\r$.
 Since both $u (\pm \infty)$ exist by Lemma~\ref{lemma4_estimate_barsol}, 
 the left-hand side is written into $a (c_ * + c) (u (\infty) - u (-\infty))$.
 On the other hand, we show that 
 the integral of the right-hand side on $\r$ is well-defined, and 
 \begin{equation}\label{eq3_integral_V}
  \int_{-\infty}^\infty (V (u (t + 1)) - V (u (t))) dt
   = V (u (\infty)) - V (u (-\infty)).
 \end{equation}
 By the same argument as in (\ref{eq2_F2_0}), we see 
 \[
 V (u (t)) - V (u (-\infty))
 = \left( \int_0^1 V' (\bar x) d \theta \right)
 \left( \int_{-\infty}^t u' (\tau) d \tau \right)
 \]
 in $t < 0$, where $\bar x \equiv \theta u (t) + (1 - \theta) u (-\infty)$.
 Then there exists $\rho > 0$ such that 
 \[
 \begin{aligned}
  & \bar \varphi_- (t) |V (u (t + 1)) - V (u (t))| \\
  & \le \bar \varphi_- (t) 
  (|V (u (t + 1)) - V (u (-\infty))| + |V (u (t)) - V (u (-\infty))|)
  \le C \b e^{- \a \b \rho}  
 \end{aligned}
 \]
 in $t \in (-\infty, -2)$ so that 
 \[
 \int_{-\infty}^{-2} (V (u (t + 1)) - V (u (t))) dt
 = \int_{-2}^{-1} V (u (t)) dt - V (u (-\infty)).
 \]
 Similarly, we have 
 \[
 \int_1^\infty (V (u (t + 1)) - V (u (t))) dt
 = - \int_1^2 V (u (t)) dt + V (u (\infty)).
 \]
 Finally, we see 
 \[
 \begin{aligned}
  \int_{-2}^1 (V (u (t + 1)) - V (u (t))) dt
  = \int_1^2 V (u (t)) dt - \int_{-2}^{-1} V (u (t)) dt, 
 \end{aligned}
 \]
 which leads to (\ref{eq3_integral_V}).

 From the arguments above, we have 
 $(c_ * + c) u (\infty) - V (u (\infty)) = (c_ * + c) u (-\infty) - V (u (-\infty))$.
 There is $\rho > 0$ be independent of $\b$ such that 
 both $u (\infty)$ and $u (- \infty)$ are less than $l - \rho$.
 Hence $|V' (x)| \le C \b e^{-\a \b \rho}$ in $x \le l - \rho$ 
 by the assumption (A$_1$).
 Since $c_* > 0$ and $c$ is close to $0$, 
 we see that $h (x) \equiv (c_ * + c) x - V (x)$ increases
 strictly monotonically in $x \le l - \rho$ because $h' (x)$ is close to 
 $c_ * + c > 0$.
 Therefore we conclude $u (\infty) = u (- \infty)$.
\end{proof}


\section{Periodic solution}
\label{sec_periodic}

The goal in this section is to construct a periodic solution $u$ 
of (\ref{eq_u}) satisfying (\ref{eq_formalsol}) and (\ref{eq_const}).
As mentioned in Introduction, the same iterative strategies as 
in the previous sections cannot be applied to (\ref{eq_u})
because the linearized equation of (\ref{eq_u}) introduced around 
the approximate profile is not invertible 
(compare (\ref{eq4_detM}) and (\ref{eq3_detM})). 
Motivated by \cite{bronsard1997volume}, we take the constraint (\ref{eq_const}) 
into account and introduce the following modified equation corresponding 
to (\ref{eq_u}) on $S \equiv \r / N \z$;
\begin{equation}\label{eq_u_periodic}
 c^2 u'' (t) + a c u' (t) = a (V (u (t + 1)) - V (u (t))) 
  + \dfrac{I_0}{N} \left( \dfrac{1}{N} \int_0^N u (t) dt - u_* \right),
\end{equation}
where 
\[
u_* \equiv (l + \eta) \dfrac{t_*}{N} + (l - \eta) \left( 1 - \dfrac{t_*}{N} \right).
\]
A parameter $I_0$ is assumed to be an arbitrary positive constant
and fixed independently of any other parameters.
Actually, we show that if there is a solution $u \in C^2 (S)$ 
of (\ref{eq_u_periodic}), it becomes a periodic solution of (\ref{eq_u})
satisfying (\ref{eq_const}).
Integrating the both sides of (\ref{eq_u_periodic}) on $(0, N)$, we have 
\[
a \int_0^N (V (u (t + 1)) - V (u (t))) dt
+ I_0 \left( \dfrac{1}{N} \int_0^N u (t) dt - u_* \right) = 0.
\]
It is easy to see that the first term above vanishes because
\[
\begin{aligned}
 \int_0^N V (u (t + 1)) dt
 & = \int_0^{N-1} V (u (t + 1)) dt + \int_{N-1}^N V (u (t + 1)) dt \\
 & = \int_1^N V (u (t)) dt + \int_0^1 V (u (t)) dt = \int_0^N V (u (t)) dt,
\end{aligned}
\]
which implies (\ref{eq_const}).
As a result, we see that (\ref{eq_u}) holds true.

Set $c_* = c_0 (\eta_*)$, $c \to c_* + c$, and $\eta \to \eta_* + \eta$, 
where $\eta_*$ is defined in Section~\ref{sec_Introduction}.
We consider (\ref{eq_u_periodic}) in $(0, t_*) \cup (t_*, N)$
($\equiv I_1 \cup I_2$) under $u (0) = u (t_*) = l$.
In order to construct a periodic solution in (\ref{eq_u_periodic}) with the period $N$,
the solution must satisfy the matching conditions at $t = 0$ and $t = t_*$ as
\begin{equation}\label{eq3_C1_matching}
 \left\{
  \begin{aligned}
   & u' (+ 0) = u' (N - 0), \\
   & u' (t_* + 0) = u' (t_* - 0).
  \end{aligned}
  \right.
\end{equation}
Define $u_1 = l - \eta_* - \eta$ and $u_2 = l + \eta_* + \eta$.
Let $u_0 (t; \eta)$ be a periodic function with period $N$, defined by 
\begin{equation}\label{eq3_up1}
 u_0 (t; \eta) 
  = u_i (t; \eta_* + \eta) \chi_R (t) + u_1 (1 - \xi_R (t + 2 R)) + u_2 \xi_R (t - R)
\end{equation}
in $- 2 R \le t \le t_* - 2 R$ and 
\begin{equation}\label{eq3_up2}
 u_0 (t + t_*; \eta) 
  = u_d (t; \eta_* + \eta) \chi_R (t) + u_2 (1 - \xi_R (t + 2 R)) + u_1 \xi_R (t - R)
\end{equation}
in $- 2 R \le t \le N - t_* - 2 R$,
where $R$ is sufficiently large and independent of $\b, N$
and $\xi_R$ is given in Section~\ref{sec_homoclinic}.
Here $\chi_R (t) = \chi (t / R)$, where $\chi (x) \in C^\infty (\r)$ is defined by 
$\chi (x) \equiv \xi (x + 2) (1 - \xi (x - 1))$ and satisfies
\[
\chi (x) = 
\left\{
\begin{aligned}
 & 1, 
 & & |x| \le 1, \\
 & 0, 
 & & |x| \ge 2.
\end{aligned}
\right.
\]
It is clear that $u_0$ is $C^1$ and twice differentiable in $t \in S$.
We note that $u_0 (0) = u_0 (t_*) = l$.

We next derive error terms between $u$ and $u_0$ in (\ref{eq_u_periodic})
and the lowest-order term of $c$ and $\eta$ close to $0$.
Let $\bar \varphi_\pm (t) = e^{\pm \k t / 2}$ for $\k \equiv a / c_*$, 
which were introduced in Section~\ref{sec_Heteroclinic}.
By substituting $u (t) = u_i (t; \eta_* + \eta) + \phi_{0i} (t)$ 
in (\ref{eq_u_periodic}) and using the same argument as in the previous sections, 
we formally obtain $\phi_{0i} (t) \equiv ( c - c_\eta \eta ) \phi_0 (t; \eta_*)$.
Moreover, we denote
\[
\zeta_{0i} \equiv \phi_{0i}' (+0) - \phi_{0i}' (-0)
= ( c - c_\eta \eta ) \dfrac{a V_0}{c_*^3} e^{-\k}.
\]
Next, we derive error terms between $u$ and $u_0$ around $t = t_*$.
Replacing $t \to t + t_*$ and substituting 
$u (t + t_*) = u_d (t; \eta_* + \eta) + \phi_{0d} (t)$ in (\ref{eq_u_periodic}), 
we formally obtain $\phi_{0d} (t) \equiv - ( c - c_\eta \eta ) \phi_0 (t; \eta_*)$.
Moreover, we denote
\[
\zeta_{0d} \equiv \phi_{0d}' (+0) - \phi_{0d}' (-0)
= - ( c - c_\eta \eta ) \dfrac{a V_0}{c_*^3} e^{-\k}.
\]
In order to emphasize the $(c, \eta)$-dependencies of 
$\phi_{0i}$, $\phi_{0d}$, $\zeta_{0i}$ and $\zeta_{0d}$, 
we may denote $\phi_{0i} (t; c, \eta)$, $\phi_{0d} (t; c, \eta)$, 
$\zeta_{0i} (c, \eta)$ and $\zeta_{0d} (c, \eta)$.

According to the arguments in Section~\ref{sec_Heteroclinic},
$\phi_{0i} (t)$ is dominant in the error between $u (t)$ and $u_i (t)$ near $t = 0$,
while $\phi_{0d} (t - t_*)$ is so between $u (t)$ and $u_d (t - t_*)$ near $t = t_*$.
We denote $c_0 = c_0 (\eta_* + \eta)$ for simplicity.
In order to evaluate the error away from $t = 0$ and $t = t_*$, 
we introduce a unique solution $\psi$ in a linear problem 
\begin{equation}\label{eq3_outersol}
 \psi'' (t) + \dfrac{a}{c_* + c} \psi' (t) = f_{o1} (t) + f_{o2} (t) 
  \quad (\equiv f_o (t)),
  \quad t \in I_1 \cup I_2
\end{equation}
with $\psi (0) = \psi (t_*) = 0$, where $f_{oj} (t)$ for $j = 1, 2$ are given by 
\[
\begin{aligned}
 f_{o1} (t) & \equiv - u_0'' (t) - \dfrac{a}{c_* + c} u_0' (t) \\
 & + \dfrac{c_0^2}{(c_* + c)^2} 
 \left( u_i'' (t; \eta_* + \eta) + \dfrac{a}{c_0} u_i' (t; \eta_* + \eta) 
 \right) \chi_R (t) \\
 & + \dfrac{c_*^2}{(c_* + c)^2} ( c - c_\eta \eta )
 \left( \dfrac{2}{c_*} u_i'' (t; \eta_*) 
 + \dfrac{a}{c_*^2} u_i' (t; \eta_*) \right) \chi_R (t) \\
 & + \dfrac{c_0^2}{(c_* + c)^2} 
 \left( u_d'' (t - t_*; \eta_* + \eta) + \dfrac{a}{c_0} u_d' (t - t_*; \eta_* + \eta) 
 \right) \chi_R (t - t_*) \\
 & + \dfrac{c_*^2}{(c_* + c)^2} ( c - c_\eta \eta ) 
 \left( \dfrac{2}{c_*} u_d'' (t - t_*; \eta_*) 
 + \dfrac{a}{c_*^2} u_d' (t - t_*; \eta_*) \right) \chi_R (t - t_*),
\end{aligned}
\]
and 
\[
\begin{aligned}
 f_{o2} (t) & \equiv 
 - (\phi_{0i} (t) \chi_R'' (t) + \phi_{0d} (t - t_*) \chi_R'' (t - t_*)) \\
 & - 2 (\phi_{0i}' (t) \chi_R' (t) + \phi_{0d}' (t - t_*) \chi_R' (t - t_*)) \\
 & - \dfrac{a}{c_* + c} 
 (\phi_{0i} (t) \chi_R' (t) + \phi_{0d} (t - t_*) \chi_R' (t - t_*)).
\end{aligned}
\]
Let $\tilde \varphi (t) \in C (S)$ be defined by 
\[
\tilde \varphi (t) = 
\left\{
\begin{aligned}
 & \min \{ \bar \varphi_+ (t), \bar \varphi_- (t - t_*), N \}, 
 & & 0 \le t \le t_*, \\
 & \min \{ \bar \varphi_+ (t - t_*), \bar \varphi_- (t - N), N \}, 
 & & t_* \le t \le N.
\end{aligned}
\right.
\]
As in the previous section, we will use $t_0 \equiv (2 \log N) / \k$.
Then it holds that 
$\bar \varphi (t_0) = \tilde \varphi (t_* - t_0) 
= \tilde \varphi (t_* + t_0) = \tilde \varphi (N - t_0) = N$.
In order to emphasize the $(c, \eta)$-dependencies of $\psi$ and $f_o$,
we may denote $\psi (t; c, \eta)$ and $f_o (t; c, \eta)$.


\begin{lemma}\label{lemma_psi}
 Let $\psi$ be the solution of (\ref{eq3_outersol}) under $\psi (0) = \psi (t_*) = 0$
 satisfying
 \[
 \psi \in C^1 (\overline{I_1}) \cap C^1 (\overline{I_2})
 \cap C^2 ([0, t_*-1]) \cap C^2 ([t_*-1, t_*])
 \cap C^2 ([t_*, N-1]) \cap C^2 ([N-1, N]).
 \]
 Suppose that $\e > 0$ is small.
 If $R$ and $N$ are sufficiently large, then there is a constant $C > 0$ 
 independent of $\e, R$ and $N$ such that $\psi$ can be estimated as
 $ |\psi (t)| \le C \e$ and $\tilde \varphi (t) |\psi' (t)| \le C \sigma (t) \e$
 in $t \in I_1 \cup I_2$ and $|c| + |\eta| \le \e$, 
 where $\sigma (t)$ is given by 
 \[
 \sigma (t) = 
 \left\{
 \begin{aligned}
  & 1,
  & & t \in (0, t_* - 2 R) \cup (t_*, N - 2 R), \\
  & \tilde \varphi (t),
  & & t \in [t_* - 2 R, t_*] \cup [N - 2 R, N].
 \end{aligned}
 \right.
 \]
 In addition, $\psi$ is Lipschitz continuous with respect to $(c, \eta)$, that is,
 \[
 \begin{aligned}
  |\psi (t; c_p, \eta_p) - \psi (t; c_q, \eta_q)|
  & \le C (|c_p - c_q| + |\eta_p - \eta_q|), \\
  \tilde \varphi (t) |\psi' (t; c_p, \eta_p) - \psi' (t; c_q, \eta_q)|
  & \le C \sigma (t) (|c_p - c_q| + |\eta_p - \eta_q|)
 \end{aligned} 
 \]
 in $t \in I_1 \cup I_2$, and $|c_r| + |\eta_r| \le \e$ for $r = p, q$.
\end{lemma}


This lemma can be shown in the same way as in Lemma~\ref{lemma4_psi}
so that we omit the details of the proof.
Extend $\psi (t)$ periodically by 
$\psi (t) \equiv \psi (t - k N)$ in $k N \le t \le (k + 1) N$ for any $k \in \z$.
We define $\nu_1$ and $\nu_2$ by $\nu_1 = \psi' (+ 0) - \psi' (- 0)$ 
and $\nu_2 = \psi' (t_* + 0) - \psi' (t_* - 0)$.
In order to emphasize the $(c, \eta)$-dependencies of $\nu_1$ and $\nu_2$,
we may denote $\nu_1 (c, \eta)$ and $\nu_2 (c, \eta)$.
Now we calculate the leading order terms of $\nu_1$ and $\nu_2$.
Throughout this section, we denote general positive constants 
independent of $\b, \e, N$ and $R$ by $C$.
We note that (\ref{eq4_fo1_int}) and (\ref{eq4_fo2_int}) hold true.
Expanding $u_d (-1; \eta_* + \eta) - u_2$ and $u_i (-1; \eta_* + \eta) - u_1$ 
as Taylor series with respect to $\eta$, we have 
$u_d (-1; \eta_* + \eta) - u_2 \approx (e^\k - 2) \eta$
and $u_i (-1; \eta_* + \eta) - u_1 \approx - (e^\k - 2) \eta$,
while
\[
\phi_{0d} (-1) = - ( c - c_\eta \eta ) \dfrac{a \eta_*}{c_*^2},
\quad \phi_{0i} (-1) = ( c - c_\eta \eta ) \dfrac{a \eta_*}{c_*^2}.
\]
Then we see
\[
\left| \nu_1 + \dfrac{a^2 \eta_*}{c_*^3} c + \dfrac{a}{c_*} \eta
\right| \le C \e^2,
\quad 
\left| \nu_2 - \left( \dfrac{a^2 \eta_*}{c_*^3} c + \dfrac{a}{c_*} \eta \right) 
\right| \le C \e^2
\]
in $|c| + |\eta| \le \e$.
Moreover, it holds that 
\[
\begin{aligned}
 & \left| \nu_1 (c_p, \eta_p) + 
 \dfrac{a^2 \eta_*}{c_*^3} c_p + \dfrac{a}{c_*} \eta_p
 - \left( \nu_1 (c_q, \eta_q)
 + \dfrac{a^2 \eta_*}{c_*^3} c_q + \dfrac{a}{c_*} \eta_q \right) \right| \\
 & \le C \e (|c_p - c_q| + |\eta_p - \eta_q|),
\end{aligned}
\]
and 
\[
\begin{aligned}
 & \left| \nu_2 (c_p, \eta_p) 
 - \left( \dfrac{a^2 \eta_*}{c_*^3} c_p + \dfrac{a}{c_*} \eta_p \right) 
 - \left( \nu_2 (c_q, \eta_q) 
 - \left( \dfrac{a^2 \eta_*}{c_*^3} c_q + \dfrac{a}{c_*} \eta_q \right)
 \right) \right| \\
 & \le C \e (|c_p - c_q| + |\eta_p - \eta_q|) 
\end{aligned}
\]
in $|c_r| + |\eta_r| \le \e$ for $r = p, q$.

Let $\tilde \psi (t) \in C^2 (\overline{I_1}) \cap C^2 (\overline{I_2})$ 
be a solution of 
\[
\tilde \psi'' + \dfrac{a}{c_* + c} \tilde \psi' 
= \dfrac{I_0}{N} \dfrac{1}{(c_* + c)^2},
\quad t \in I_1 \cup I_2
\]
under $\tilde \psi (0) = \tilde \psi (t_*) = 0$.
It is easy to have the explicit form of $\tilde \psi$ as
\begin{equation}\label{eq3_tildepsi}
 \tilde \psi (t) = 
  \left\{
   \begin{aligned}
    & \dfrac{I_0}{a (c_* + c) N} 
    \left( t - t_* \dfrac{1 - e^{- \hat \k t}}{1 - e^{- \hat \k t_*}} \right),
    & & t \in I_1, \\
    & \dfrac{I_0}{a (c_* + c) N} 
    \left( t - t_* - (N - t_*) 
    \dfrac{1 - e^{- \hat \k (t - t_*)}}{1 - e^{- \hat \k (N - t_*)}} \right),
    & & t \in I_2,
   \end{aligned}
  \right.
\end{equation}
where $\hat \k = a / (c_* + c)$ as introduced in the proof of Lemma~\ref{lemma_psi}.
By direct calculations, we see that $| \tilde \psi (t) | \le C$ 
and $\tilde \varphi (t) | \tilde \psi' (t) | \le C$ in $t \in I_1 \cup I_2$.
Since $\tilde \psi'$ is not continuous at $t = 0, t_*$, we define
$\tilde \nu_1 \equiv \tilde \psi' (+ 0) - \tilde \psi' (- 0)$
and $\tilde \nu_2 \equiv \tilde \psi' (t_* + 0) - \tilde \psi' (t_* - 0)$.
We easily estimate $\tilde \nu_1$ and $\tilde \nu_2$ as
\[
\left| \tilde \nu_1 + \dfrac{I_0}{c_*^2} \dfrac{t_*}{N} \right| \le C \e,
\quad \left| \tilde \nu_2 + \dfrac{I_0}{c_*^2} \dfrac{N - t_*}{N} \right| \le C \e
\]
in $|c| \le \e$ if $N$ is sufficiently large.
In order to emphasize the $c$-dependencies of $\tilde \psi$, 
$\tilde \nu_1$ and $\tilde \nu_2$,
we may denote $\tilde \psi (t; c)$, $\tilde \nu_1 (c)$ and $\tilde \nu_2 (c)$.
Note that $\tilde \psi$, $\tilde \nu_1$ and $\tilde \nu_2$ are independent of $\eta$.
Then it follows from (\ref{eq3_tildepsi}) that 
$| \tilde \psi (t; c_p) - \tilde \psi (t; c_q) | \le C |c_p - c_q|$,
$\tilde \varphi (t) | \tilde \psi' (t; c_p) - \tilde \psi' (t; c_q) | 
\le C |c_p - c_q|$,
$| \tilde \nu_1 (c_p) - \tilde \nu_1 (c_q) | \le C |c_p - c_q|$,
and 
$| \tilde \nu_2 (c_p) - \tilde \nu_2 (c_q) | \le C |c_p - c_q|$
in $t \in I_1 \cup I_2$ and $|c_r| \le \e$ for $r = p, q$.
Finally, we extend $\tilde \psi (t)$ periodically to the whole line 
in the same way as in the case of $\psi (t)$.

For $U \equiv (\phi, c, \eta)$, we define $\tilde u (t) = \tilde u (t; U)$ by 
\[
\begin{aligned}
 \tilde u (t) = u_0 (t) + \phi_{0i} (t) \chi_R (t) 
 + \phi_{0d} (t - t_*) \chi_R (t - t_*) + \psi (t) 
 + I (U) \tilde \psi (t) + \phi (t)
\end{aligned}
\]
in $[- 2 R, N - 2 R]$, where $I (U) \equiv \hat I (U) / \tilde I (c)$
for 
\[
\begin{aligned}
 \hat I (U) & \equiv 
 \dfrac{1}{N} \int_0^N (u_0 (t) + \phi_{0i} (t) \chi_R (t) 
 + \phi_{0d} (t - t_*) \chi_R (t - t_*) + \psi (t) + \phi (t)) dt - u_*, \\
 \tilde I (c) & \equiv 1 - \dfrac{1}{N} \int_0^N \tilde \psi (t) dt.
\end{aligned}
\]
This definition leads to 
\[
I (U) = \dfrac{1}{N} \int_0^N \tilde u (t) dt - u_*.
\]
By means of a natural extension, $\tilde u$ is considered as a periodic function on $S$.
In order to prove Theorem~\ref{thm_3}, it is sufficient to find 
$U = (\phi, c, \eta) \in W_N$ 
such that $\tilde u$ satisfies (\ref{eq_u_periodic}) and (\ref{eq3_C1_matching}),
where $W_N$ is defined by 
\[
W_N \equiv \{ U = (\phi, c, \eta) \in C (S) \times \r^2 \mid 
\phi \in C^1 (\overline{I_1}) \cap C^1 (\overline{I_2}),
\quad \phi (0) = \phi (t_*) = 0 \},
\]
equipped with the norm $\| \cdot \|_{W_N}$ defined by 
$\| U \|_{W_N} \equiv \| \phi \| + |c| + |\eta|$
and $\| \phi \| \equiv 
(\| \phi' \tilde \varphi \|_{I_1} + \| \phi' \tilde \varphi \|_{I_2}) / \sqrt{\e}$.
The parameter $\e > 0$ is assumed to be sufficiently small 
and independent of $\b, N, R$.
Contrary to $W$ in Section~\ref{sec_Heteroclinic}
and $W_\infty$ in Section~\ref{sec_homoclinic}, the present norm $\| \cdot \|$ 
includes an additional factor of $1/\sqrt{\e}$ to ensure that the influence of 
$\phi$ on $I (U)$ remains sufficiently small.
Clearly, $W_N$ is a Banach space.
Actually, we prove $\| \phi \|_S \le C \sqrt{\e} \| \phi \|$.
Since (\ref{ea2_poincare_ineq}) holds true by $\phi (0) = 0$, we have 
\[
|\phi (t)| \le \sqrt{\e} \| \phi \| \int_0^t 
\max \{ \varphi_- (s), \dfrac{1}{N} \} ds \le C \sqrt{\e} \| \phi \|
\]
in $t \in [0, t_*/2]$.
In other intervals, we can also estimate $\phi$ in similar ways.
Moreover, we define a closed ball in $W_N$ with the radius $\e$ by 
$W_N^\e \equiv \{ U \in W_N \mid \| U \|_{W_N} \le \e \}$.

We estimate $I (U)$ in $U \in W_N^\e$.


\begin{lemma}\label{lemma3_IU}
 If $\e$ is sufficiently small and $N$ is sufficiently large, 
 then $\tilde I (c) \ge 1$ in $|c| \le \e$.
 Moreover, there is $C > 0$ independent of $\e, N$ and $R$ such that 
 \begin{equation}\label{eq3_Ic_Lipschitz}
  |\tilde I (c_p) - \tilde I (c_q)| \le C |c_p - c_q|
 \end{equation}
 in $|c_r| \le \e$ for $r = p, q$.
 Similarly, it holds that 
 \begin{equation}\label{eq3_IU}
  \left| I (U) - \dfrac{1}{\tilde I (0)}
   \left( \dfrac{a \eta_*}{c_*^2} c + \eta \right) \dfrac{N - 2 t_*}{N} 
  \right| \le C \e^{3/2}
 \end{equation}
 in $U \in W_N^\e$ and 
 \begin{equation}\label{eq3_IU_Lipschitz}
  \begin{aligned}
   & \left| I (U_p) - \dfrac{1}{\tilde I (0)}
   \left( \dfrac{a \eta_*}{c_*^2} c_p + \eta_p \right) \dfrac{N - 2 t_*}{N} 
   - \left( I (U_q) - \dfrac{1}{\tilde I (0)}
   \left( \dfrac{a \eta_*}{c_*^2} c_q + \eta_q \right) \dfrac{N - 2 t_*}{N} 
   \right) \right| \\
   & \le C \sqrt{\e} \| U_p - U_q \|   
  \end{aligned}
 \end{equation}
 in $U_r \in W_N^\e$ for $r = p, q$.
\end{lemma}


\begin{proof}
 It follows from (\ref{eq3_tildepsi}) and the assumption that 
 \[
 \left| \dfrac{1}{N} \int_0^N \tilde \psi (t) dt
 + \dfrac{I_0}{2 a c_*} \dfrac{(N - t_*)^2 + t_*^2}{N^2}
 \right| \le C \e,
 \]
 from which 
 \[
 \dfrac{1}{N} \int_0^N \tilde \psi (t) dt < 0
 \]
 so that $\tilde I (c) \ge 1$.
 Similarly, we easily prove (\ref{eq3_Ic_Lipschitz}).

 If $N$ is sufficiently large, we have  
 \[
 \left| \dfrac{1}{N} \int_0^N u_0 (t) dt - u_* \right| 
 \le C \dfrac{R}{N} \le C \e^2
 \]
 by (\ref{eq_u_infty}), (\ref{eq3_up1}) and (\ref{eq3_up2}).
 From $\| \phi_{0i} \|_{\r} \le C \e$ and $\| \phi_{0d} \|_{\r} \le C \e$, we see 
 \[
 \left| \dfrac{1}{N} \int_{-2R}^{2R} \phi_{0i} (t) \chi_R (t) dt \right| \le C \e^2,
 \quad \left| \dfrac{1}{N} \int_{t_* - 2R}^{t_* + 2R}
 \phi_{0d} (t - t_*) \chi_R (t - t_*) dt \right| \le C \e^2,
 \]
 while 
 \[
 \left| \dfrac{1}{N} \int_0^N \phi (t) dt \right| \le C \e^{3/2}
 \]
 by $\| \phi \| \le \e$.
 On the other hand, we see
 \[
 \left| \dfrac{1}{N} \int_0^N \psi (t) dt
 - \left( \dfrac{a \eta_*}{c_*^2} c + \eta \right) \dfrac{N - 2 t_*}{N} \right|
 \le C \e^2.
 \]
 Hence we conclude (\ref{eq3_IU}).
 In the same way as above, we obtain (\ref{eq3_IU_Lipschitz}).
\end{proof}


In the same way as in Section~\ref{sec_Heteroclinic},
we substitute $u (t) = \tilde u (t)$ into (\ref{eq_u_periodic})
and derive a linear equation, given by 
\begin{equation}\label{eq3_phi_modify}
 \bar \phi'' (t) + \k \bar \phi' (t) = F (t; U),
  \quad t \in I_1 \cup I_2.
\end{equation}
Here the function $F$ in the right-hand side is decomposed into two parts as
\[
F (t; U) \equiv F_1 (t; U) + \dfrac{a}{(c_* + c)^2} F_2 (t; U),
\]
where
\[
\begin{aligned}
 F_1 (t; U) 
 & \equiv - \left( \dfrac{a}{c_* + c} - \dfrac{a}{c_*} \right) \phi' (t) \\
 & - \left( 1 - \dfrac{c_*^2}{(c_* + c)^2} \right) 
 ( \phi_{0i}'' (t) \chi_R (t) + \phi_{0d}'' (t - t_*) \chi_R (t - t_*) )\\
 &  - \dfrac{a c}{(c_* + c)^2}
 (\phi_{0i}' (t) \chi_R (t) + \phi_{0d}' (t - t_*) \chi_R (t - t_*)) \\
\end{aligned}
\]
and
\[
\begin{aligned}
 F_2 (t; U) 
 & \equiv V (\tilde u (t + 1)) - V (\tilde u (t))
 - V_0 (H (u_i (t + 1) - l) - H (u_i (t) - l))) \chi_R (t) \\
 & - V_0 (H (u_d (t - t_* + 1) - l) - H (u_d (t - t_*) - l))) \chi_R (t - t_*)).
\end{aligned}
\]
Similarly, we substitute $\tilde u$ into (\ref{eq3_C1_matching}), obtaining
\begin{equation}\label{eq3_C1_matching_phi}
 \left\{
 \begin{aligned}
  & M_{11} \bar c + M_{12} \bar \eta 
  + \bar \phi' (+ 0) - \bar \phi' (N - 0) = G_1 (U), \\
  & M_{21} \bar c + M_{22} \bar \eta
  + \bar \phi' (t_* + 0) - \bar \phi' (t_* - 0) = G_2 (U),
 \end{aligned}
 \right.
\end{equation}
where $M_{ij}$ and $G_i (U)$ for $i, j = 1, 2$ are given by 
\[
\begin{aligned}
 & M_{11} 
 \equiv \dfrac{a V_0}{c_*^3} \left( 1 - \k - \dfrac{I_0}{2 \tilde I (0) c_*^2} 
 \dfrac{N - 2 t_*}{N} \dfrac{t_*}{N} \right), \\
 & M_{21} \equiv
 - \dfrac{a V_0}{c_*^3} \left( 1 - \k + \dfrac{I_0}{2 \tilde I (0) c_*^2} 
 \dfrac{N - 2 t_*}{N} \dfrac{N - t_*}{N} \right), \\
 & M_{12} \equiv - \dfrac{I_0}{\tilde I (0) c_*^2} 
 \dfrac{N - 2 t_*}{N} \dfrac{t_*}{N},
 \quad M_{22} \equiv - \dfrac{I_0}{\tilde I (0) c_*^2} 
 \dfrac{N - 2 t_*}{N} \dfrac{N - t_*}{N},
\end{aligned}
\]
\[
\begin{aligned}
 G_1 (U) & \equiv 
 - \left( \nu_1 + \dfrac{a^2 \eta_*}{c_*^3} c + \dfrac{a}{c_*} \eta \right)
 - I (U) \left( \tilde \nu_1 + \dfrac{I_0}{c_*^2} \dfrac{t_*}{N} \right) \\
 & + \left( I (U) - \dfrac{1}{\tilde I (0)} 
 \left( \dfrac{a \eta_*}{c_*^2} c + \eta \right) \dfrac{N - 2 t_*}{N} \right)
 \dfrac{I_0}{c_*^2} \dfrac{t_*}{N}, 
\end{aligned}
\]
and 
\[
\begin{aligned}
 G_2 (U) & \equiv 
 - \left( \nu_2 - \dfrac{a^2 \eta_*}{c_*^3} c - \dfrac{a}{c_*} \eta \right)
 - I (U) \left( \tilde \nu_2 + \dfrac{I_0}{c_*^2} \dfrac{N - t_*}{N} \right) \\
 & + \left( I (U) - \dfrac{1}{\tilde I (0)}
 \left( c \dfrac{a \eta_*}{c_*^2} + \eta \right) \dfrac{N - 2 t_*}{N} \right)
 \dfrac{I_0}{c_*^2} \dfrac{N - t_*}{N}. 
\end{aligned}
\]
We compute the determinant of a matrix $M \equiv (M_{ij})$ for $i, j = 1, 2$ as
\begin{equation}\label{eq3_detM}
 \begin{aligned}
  \det M 
  & = - \dfrac{a V_0}{c_*^3} \left( 1 - \k - \dfrac{I_0}{2 \tilde I (0) c_*^2} 
  \dfrac{N - 2 t_*}{N} \dfrac{t_*}{N} \right)
  \dfrac{I_0}{\tilde I (0) c_*^2} \dfrac{N - 2 t_*}{N} \dfrac{N - t_*}{N} \\
  & - \dfrac{a V_0}{c_*^3} \left( 1 - \k + \dfrac{I_0}{2 \tilde I (0) c_*^2} 
  \dfrac{N - 2 t_*}{N} \dfrac{N - t_*}{N} \right) 
  \dfrac{I_0}{\tilde I (0) c_*^2} \dfrac{N - 2 t_*}{N} \dfrac{t_*}{N} \\
  & = - \dfrac{a V_0}{c_*^3} ( 1 - \k ) 
  \dfrac{I_0}{\tilde I (0) c_*^2} \dfrac{N - 2 t_*}{N} \neq 0.
 \end{aligned} 
\end{equation}

We think of (\ref{eq3_phi_modify}) and (\ref{eq3_C1_matching_phi})
as a linear system with respect to $\bar U = (\bar \phi, \bar c, \bar \eta)$.
To begin with, we focus on the left-hand sides of 
(\ref{eq3_phi_modify}) and (\ref{eq3_C1_matching_phi}), 
and study a linear problem given by 
\begin{equation}\label{eq3_L}
 \left\{
  \begin{aligned}
   & \phi'' (t) + \k \phi' (t) = f, 
   \quad t \in I_1 \cup I_2, \\
   & M_{11} c + M_{12} \eta + \phi' (+ 0) - \phi' (N - 0) = g_1, \\
   & M_{21} c + M_{22} \eta + \phi' (t_* + 0) - \phi' (t_* - 0) = g_2,
  \end{aligned}
 \right.
\end{equation}
where we omit the bars above the letters in the expressions for simplicity.
By the same arguments as in the previous section, 
we find a solution $U \in W_N$ for $h = (f, g_1, g_2) \in W_N'$, 
where $W_N'$ is defined by 
\[
W_N' \equiv (C ([0, t_* - 1]) \cap C ([t_* - 1, t_*]) \cap C ([t_*, N - 1]) 
\cap C ([N - 1 , N])) \times \r^2.
\]
By the same argument as in Lemma~\ref{lemma4_linearsol}, 
we can prove the following lemma.
So we omit the details of the proof.


\begin{lemma}\label{lemma_linearsol_per}
 For any $h = (f, g_1, g_2) \in W_N'$, 
 there exists a unique solution $U = (\phi, c, \eta) \in W_N$ 
 of (\ref{eq3_L}) such that $\| \phi \| \le C \| f \tilde \varphi \|_S / \sqrt{\e}$
 and $|c| + |\eta| \le C (\| f \tilde \varphi \|_S + |g_1| + |g_2|)$
 for some $C > 0$ independent of $\b, \e, N$ and $R$.
\end{lemma}


Lemma~\ref{lemma_linearsol_per} leads to the existence of a solution $\bar U$ 
in (\ref{eq3_phi_modify}) and (\ref{eq3_C1_matching_phi})
such that $\bar U \in W_N$, which implies that we can define a map $T$
by $\bar U = T (U)$.
Actually, we can show that $T$ is a contraction mapping on $W_N^\e$,
which implies that $T$ has a fixed point $U \in W_N^\e$, 
and leads to Theorem~\ref{thm_3}.
The following two lemmas can be shown in the same ways 
as in Lemmas~\ref{lemma4_estimate_barsol} and \ref{lemma4_Lipschitz_T_homoclinic}.
We omit the details of the proofs.


\begin{lemma}\label{lemma_estimate_barsol}
 If $U \in W_N^\e$, then there exists a unique solution $\bar U \in W_N^\e$ 
 in (\ref{eq3_phi_modify}) and (\ref{eq3_C1_matching_phi}).
\end{lemma}


\begin{lemma}\label{lemma_Lipschitz_T_periodic}
 Suppose that $\e > 0$ is small and $R$ is large.
 If $\b$ and $N$ are sufficiently large depending on $\e, R$, 
 then there exists $C > 0$ independent of $\b, \e, N$ and $R$ such that 
 \[
 \| T (U_p) - T (U_q) \|_{W_N} \le C \sqrt{\e} \| U_p - U_q \|_{W_N}
 \]
 in $U_p, U_q \in W_N^\e$.
 In particular, $T$ is a contraction mapping and has a fixed point $U \in W_N^\e$.
\end{lemma}


\begin{figure}[h]
 \begin{center}
  \begin{tabular}{cc} 
   \includegraphics[width=0.45\hsize]{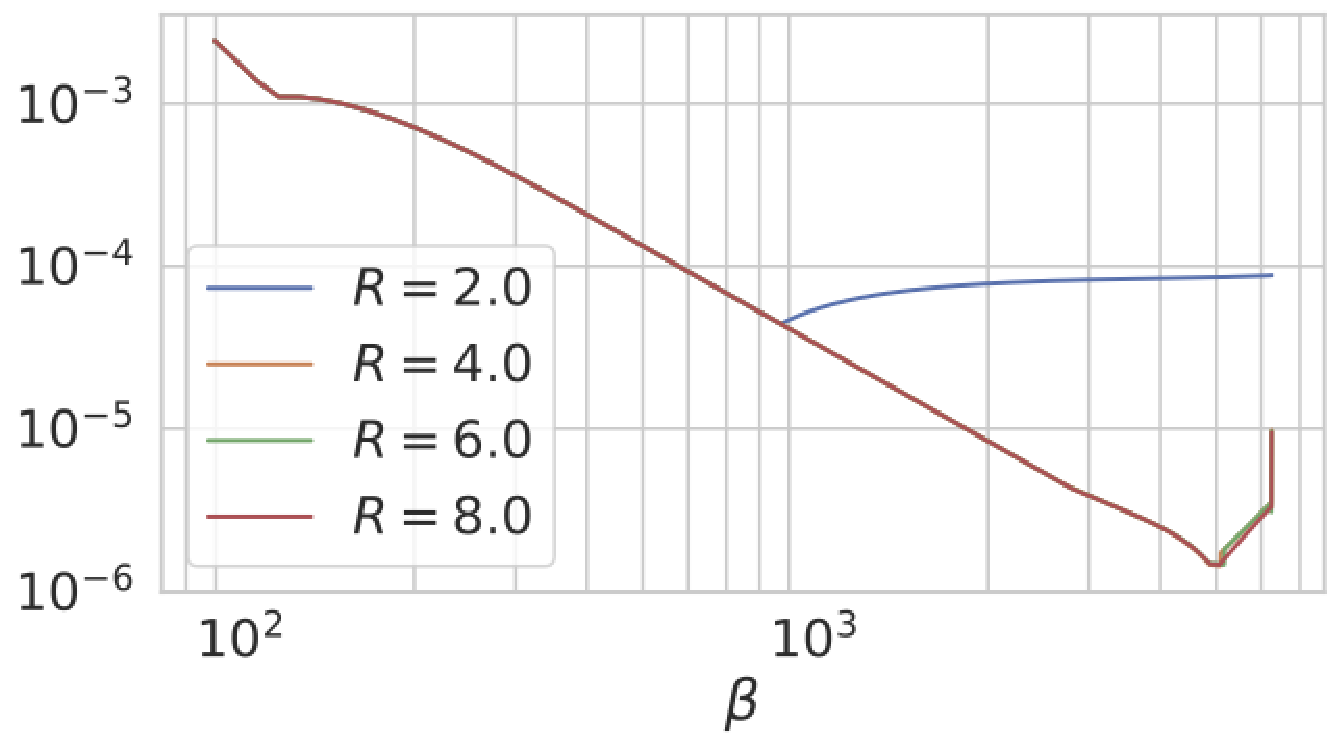} & 
   \includegraphics[width=0.45\hsize]{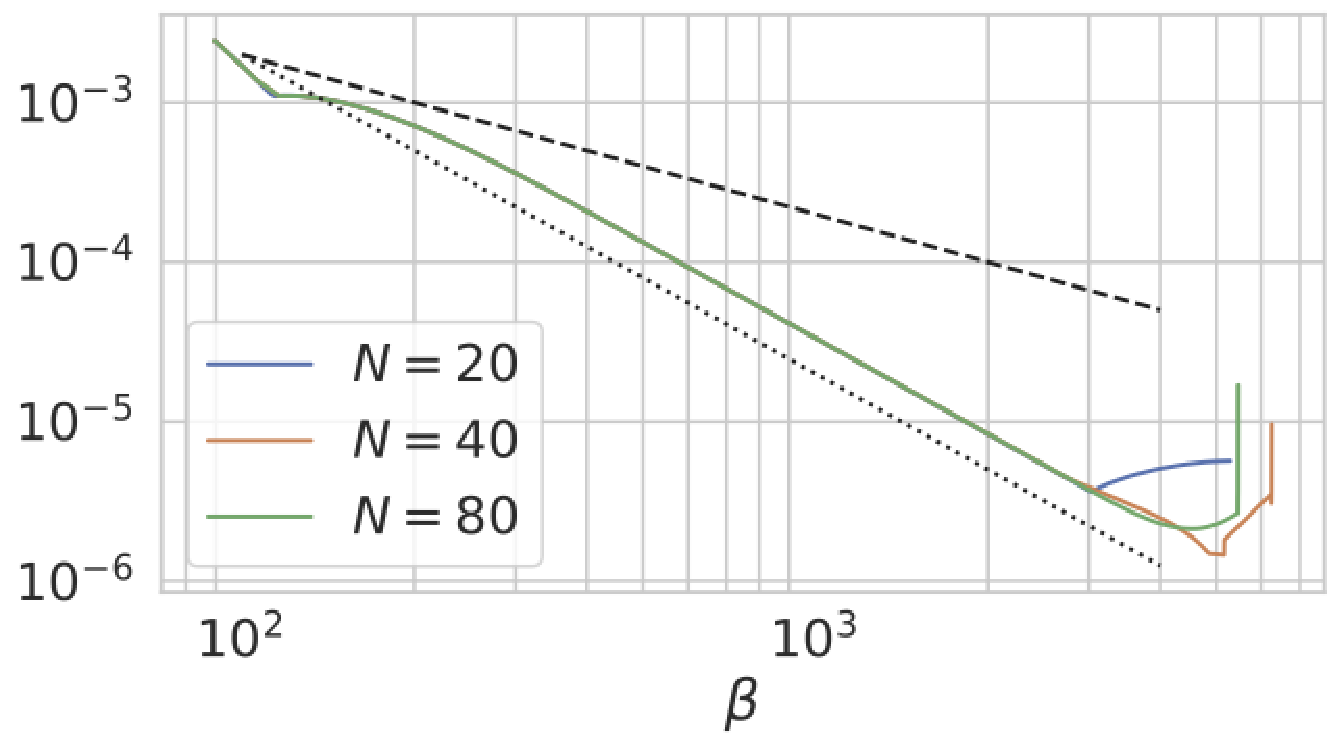} \\
   (a) & (b)
  \end{tabular}
  \caption{
  Approximation errors as functions of $\beta$ for several values of $R$ and $N$.
  (a) $R = 2, 4, 6, 8$ with $(N, L) = (40, 1.0)$, and 
  (b) $N = 20, 40, 80$ with $R = 3 N / 20$. 
  Dashed and dotted lines in (b) represent $O(\beta^{-1})$ and $O(\beta^{-2})$ slopes, 
  respectively.
  }
  \label{fig:errors}
 \end{center}
\end{figure}


Finally, we conducted a numerical comparison between the periodic traveling 
wave solution of (\ref{eq_u}), and the function $u_0$ defined 
by (\ref{eq3_up1}) and (\ref{eq3_up2}).
For the OV function in (\ref{eq_OVfunction_tanh}), 
we used the same parameter values in Fig.~\ref{fig_periodic}, 
namely $V_0 = 0.0336$, $l = 0.025$, and $M = 0.913$ (\cite{bando1995phenomenological}),
whereas $\b$ varied from $2/0.0223$ to $6\times 10^3$ as a parameter.
We selected $(N, L) = (20, 0.5), (40, 1.0)$, and $(80, 2.0)$
such that the two transition layers are separated by approximately half the domain.
We approximated the solution to (\ref{eq_u}) using the Fourier spectral method 
\cite{boyd2001chebyshev},
where the nonlinear terms were evaluated pseudospectrally, with zero-padding 
(doubling the number of modes) to reduce aliasing errors.
We applied the predictor-corrector method (\cite{kuznetsov2023numerical}) 
with an integral phase condition for numerical continuation
of periodic solutions of (\ref{eq_u}) with respect to $\beta$.
The approximation error was measured using the maximum norm in the domain.

Fig.~\ref{fig:errors} (a) presents the dependence of the error on $R$ 
for $(N, L) = (40, 1.0)$.
The plots for $R=4, 6$, and $8$ are nearly identical, 
whereas that for $R=2$ separates near $\beta = 10^3$.
This is because $R$ must be sufficiently large so that the transition layers 
are appropriately approximated.
The approximation error decreases monotonically with $\beta$. 
However, it increases for $\beta > 3\times 10^3$.
This increase in error is caused by the truncation error.
Since the profile tends to resemble a square wave as $\beta$ increases, 
additional Fourier modes are needed to resolve this problem.
The results for $(N, L)=(20, 0.5)$ and $(80, 2.0)$ were similar.

Fig.~\ref{fig:errors} (b) shows the numerical results 
for $N=20, 40$, and $80$ with $R = 3N / 20$.
In this figure, the dashed and dotted lines are proportional to $\beta^{-1}$ 
and $\beta^{-2}$, respectively.
This result suggests that the approximation error is $O(\beta^{-p})$ when $1<p \le 2$.


\section{Discussion}

In this study, we established the existence of heteroclinic,
homoclinic, and periodic solutions of the difference-differential 
equation (\ref{eq_u}), derived from the OV model (\ref{eq_OV}).
These results provide a mathematically rigorous foundation for complex
solution structures that were observed numerically in microscopic
traffic flow models.

Our analysis is based on a perturbation approach around the singular limiting
profiles associated with the step function limit of the OV function.
Although the approximate solution introduced by \cite{sugiyama2023dynamics} 
plays a crucial role as a leading-order description, 
it is not sufficient on its own; further insight is needed to complete the proofs.
A key step in our analysis is the detailed investigation of the linearized
equations for approximate solutions.
In particular, the careful construction and estimation of the outer solutions
$\psi$ and $\tilde \psi$ in Sections~\ref{sec_homoclinic} and \ref{sec_periodic}
are essential for controlling the error terms in an iterative scheme.
This analysis constitutes a central technical contribution of this study.

To construct the periodic solutions, we introduced the modified equation
(\ref{eq_u_periodic}) by imposing the constraint (\ref{eq_const}), 
which reflects the intrinsic properties of (\ref{eq_OV}), 
namely the conservation of total density over time.
However, constructing periodic solutions using perturbation methods without 
such constraints is difficult owing to the non-invertibility 
of the associated linearized operator (see (\ref{eq3_detM})).
Such difficulties are closely related to the phenomena observed 
in the reaction-diffusion equations, most notably the Allen--Cahn equation,
with small diffusion coefficients \cite{bronsard1997volume}.
By incorporating the constraint, we successfully constructed 
the periodic solutions whose structure comprised well-separated transition layers.

By contrast, we demonstrated that for an arbitrarily prescribed separation $t_*$, 
homoclinic solutions satisfying (\ref{eq_formalsol})
and (\ref{eq_formalsol2}) can be constructed without imposing any additional constraint.
Such degrees of freedom do not occur in the Allen--Cahn equation,
highlighting the fundamental difference between difference-differential equations
arising from microscopic traffic models and classical reaction-diffusion equations.
Through the explicit construction of the characteristic solutions,
both the similarities and essential differences between the two classes of
the equations become apparent.

Although this study focuses on a specific difference-differential equation,
naturally derived from the OV model, the analytical framework developed here
is expected to be applicable to more general self-driven particle systems
with nearest-neighbor interactions.
In particular, extending the proposed approach to a broader class of
microscopic traffic models remains an important direction for future research.
Overall, the results obtained in this paper deepen the theoretical understanding
of the mathematical structure underlying the difference-differential equations
in traffic flow modeling.
They provide a rigorous basis for further analytical
studies that involve complex spatiotemporal patterns in vehicular dynamics.

\section*{Acknowledgments}

We would like to thank Editage (www.editage.jp) for English language editing.
This work was funded by JSPS KAKENHI (Grant Numbers JP20K03757 and JP23K25786)
and the Research Institute for Mathematical Sciences,
an International Joint Usage/Research Center located in Kyoto University.


\bibliographystyle{abbrv}
\bibliography{reference}


\end{document}